\documentclass[11pt]{amsart}  
\usepackage[english]{babel}
\usepackage{amsbsy,amsmath,amsthm,amssymb,color,verbatim}
\usepackage[backrefs]{amsrefs}
\usepackage{float}
\usepackage{fullpage,cancel,hyperref}
\usepackage{thmtools}
\usepackage{thm-restate}
\usepackage{hyperref}
\usepackage{cleveref}

\usepackage{float}
\usepackage{wrapfig}
\usepackage[dvipsnames]{xcolor}
\usepackage{graphicx}
\usepackage{subcaption}
\usepackage{multicol}
\usepackage{multirow}
\usepackage{longtable}
\usepackage{tikz-cd}
\usepackage{tikz-3dplot}
\usepackage{enumitem}
\usepackage{listings}
\usepackage{makecell}
\setcellgapes{4pt}
\usepackage{supertabular}
\usepackage{tikz}
\usetikzlibrary{decorations.pathreplacing}
\usepackage{tkz-graph}
\tikzstyle{vertex}=[circle, draw, inner sep=0pt, minimum size=4pt]

\tikzstyle{vtx}=[circle, draw, inner sep=0pt, minimum size=8pt]

\definecolor{goldenpoppy}{rgb}{0.99, 0.76, 0.0}
\definecolor{darkgreen}{cmyk}{.9,0,.9,.2}
\definecolor{midgray}{gray}{0.60}
\definecolor{lightgray}{gray}{0.90}
\definecolor{lmgray}{gray}{0.70}
\usepackage{colortbl}

\def\ZZ{{\mathbb Z}}

\allowdisplaybreaks

\makeatletter
\newcommand{\addresseshere}{%
  \enddoc@text\let\enddoc@text\relax
}
\makeatother

\theoremstyle{plain}

\newtheorem{theorem}{Theorem}[section]

\newtheorem{lemma}{Lemma}[section] 
\newtheorem{corollary}{Corollary}[section]

\theoremstyle{definition}
\newtheorem{definition}{Definition}[section]
\newtheorem{remark}{Remark}[section]
\newtheorem{question}{Question}[section]

\title{On $(t,r)$ broadcast domination of certain grid graphs}

\author{Natasha Crepeau}
\address{Department of Mathematics, University of Washington, United States}
\email{\textcolor{blue}{\href{mailto:ncrepeau@uw.edu}{ncrepeau@uw.edu}}}

\author{Pamela E. Harris}
\address{Department of Mathematical Sciences, University of Wisconsin-Milwaukee, United States}
\email{\textcolor{blue}{\href{mailto:peharris@uwm.edu}{peharris@uwm.edu}}}

\author{Sean Hays}
\address{Department of Mathematics, University of Alabama, United States}
\email{\textcolor{blue}{\href{mailto:sphays@crimson.ua.edu}{sphays@crimson.ua.edu}}}

\author{Marissa Loving}
\address{Department of Mathematics, University of Wisconsin-Madison, United States}
\email{\textcolor{blue}{\href{mailto:mloving2@wisc.edu}{mloving2@wisc.edu}}}

\author{Joseph Rennie}
\address{Department of Mathematics, University of Illinois at Urbana-Champaign, United States}
\email{\textcolor{blue}{\href{mailto:rennie2@illinois.edu}{rennie2@illinois.edu}}}

\author{Gordon Rojas Kirby}
\address{School of Mathematics and Statistical Sciences, Arizona State University, United States}
\email{\textcolor{blue}{\href{mailto:girkirby@asu.edu}{girkirby@asu.edu}}}

\author{Alexandro Vasquez}
\address{Department of Mathematics, Manhattan College, United States}
\email{\textcolor{blue}{\href{mailto:alexvsqz98@gmail.com}{alexvsqz98@gmail.com}}}


\begin{document}
\maketitle

\begin{abstract} 
Let $G=( V(G), E(G) )$ be a connected graph with vertex set $V(G)$ and edge set $E(G)$. We say a subset $D$ of $V(G)$ dominates $G$ if every vertex in $V \setminus D$ is adjacent to a vertex in $D$. A generalization of this concept is $(t,r)$ broadcast domination. We designate certain vertices to be towers of signal strength $t$, which send out signal to neighboring vertices with signal strength decaying linearly as the signal traverses the edges of the graph. We let $\mathbb{T}$ be the set of all towers, and we define the signal received by a vertex $v\in V(G)$ from all towers $w \in \mathbb T$ to be \mbox{$f(v)=\sum_{w\in \mathbb{T}}max(0,t-d(v,w))$}. Blessing--Insko--Johnson--Mauretour defined a $(t,r)$ \textit{broadcast dominating set}, or a $(t,r) $ \textit{broadcast}, on $G$ as a set $\mathbb{T} \subseteq V(G) $ such that $f(v)\geq r$ for all $v\in V(G)$. The minimum cardinality of a $(t, r)$ broadcast on $G$ is called the $(t, r)$ \textbf{broadcast domination number} of $G$. 
In this paper, we present our research on the $(t,r)$ broadcast domination number for certain graphs including paths, grid graphs, the slant lattice, and the king's lattice.
\end{abstract}

\section{Introduction}

Structures such as cell phone towers radiate signal to their surroundings.  As the distance from the tower increases, the signal becomes weaker. Eventually, areas sufficiently far from the tower will not receive any signal. As a result, the towers must be placed strategically to provide signal to all areas while using as few towers as possible. The concept of $(t,r)$ broadcasting models this scenario. Let $G=(V(G),E(G))$ be a graph and $\mathbb{T}\subseteq V(G)$ be a set of broadcasting vertices, or towers. Each broadcasting vertex is assigned a finite \textbf{transmission strength}, $t$, to send signal to nearby vertices. We define the \textbf{signal} received by a vertex $v\in V(G)$ as 
\[ f(v)=\sum_{w\in \mathbb{T}} \max\{0, (t-d(v,w))\},\] 
where $w$ is a tower and $d(v,w)$ denotes the distance between $v$ and $w$. We use the term {\bf broadcast zone} to describe the neighborhood $N_{t-1}(w):= \{v \in G: d(v, w) \leq t-1 \}$ that receives signal from a given tower $w$. We also use the term {\bf overlap} to describe vertices in the broadcast zone of more than one tower.

In 2014, Blessing--Insko--Johnson--Mauretour \cite{gridgraph} introduced the notion of {\bf $(t,r)$ broadcast domination}. A collection of vertices $D \subseteq V(G)$ is a {\bf $(t, r)$ broadcast dominating set} for $G$ if every $w \in D$ is a broadcast tower of strength $t$ and every vertex $v \in V(G)$ receives a signal of at least $r$. The $(t,r)$ \textbf{broadcast domination number} of $G$, denoted $\gamma_{t,r}(G)$, is the minimum cardinality of a $(t,r)$ broadcast dominating set for $G$. An \textbf{efficient broadcast} is a $(t,r)$ broadcast dominating set that minimizes wasted signal in the sense that every vertex in the overlap between multiple towers receives a signal of strength exactly $r$. Throughout this paper we refer to the $(t,r)$ broadcast domination number as the domination number, and the $(t,r)$ broadcast dominating set as the dominating set.

In this paper, we establish an explicit formula for the $(t,r)$ broadcast domination number on paths with $n$ vertices, which we denote $P_n$. An algorithm (written in SageMath) to generate a $(t,r)$ broadcast dominating set of $P_n$ for given $t, r$, and $n$ is provided in \Cref{ap:algorithm}.

\begin{restatable}[]{thm}{trpath}
\label{thm:trpath}
If $n \geq 1$ and $t \geq r \geq 1$, then \[\gamma_{t,r}(P_n)= \left \lceil \frac{n+(r-1)}{2t - r} \right \rceil.\]
\end{restatable}

We examine grid graphs with $m$ rows and $n$ columns, denoted $G_{m,n}$. There are several notable previous results on the $(t,r)$ broadcast domination of $m \times n$ grid graphs for small values of $m$. In particular, Blessing-Insko-Johnson-Mauretour \cite{gridgraph} gave formulas for $\gamma_{t,r}(G_{3,n})$ and $\gamma_{t,r}(G_{4,n})$ when $(t,r)\in \{(3,1),(2,2),(3,2)\}$. In 2017, Randolph \cite{randolph} proved an upper bound of $\gamma_{t,2}(G_{m,n})$ for $t > 2$. We establish an upper bound for the $(t,r)$ broadcast domination number of an $m \times n$ grid graph for $n \geq 2t-r-(m-2)$ and $2t-r > m-1$ by developing the notion of a $(t,r)$ \textbf{starting block} (defined in Section 3), which is used throughout the paper. 

\begin{restatable}[]{thm}{mxntr}

\label{thm:mxntr}
If $m \geq 2$, $n \geq 2t - r -(m-2)$, and $2t - r > m-1$, then \[\gamma_{t,r}(G_{m,n}) \leq 2 + \left \lceil \frac{n- (2t - r - (m-2))}{(2t - r - (m-2))-1} \right \rceil.\]
\end{restatable}

We also consider the $(t,r)$ broadcast domination number of $3$-dimensional grid graphs of size $m \times n \times k$, whose vertices are subsets of $\mathbb Z^3$ with edges joining vertices that are distance $1$ apart.

\begin{restatable}[]{thm}{upperbound}
\label{thm:upperbound}
Let $G_{m,n,k}$ be a 3D grid graph, and let $B$ be the number of $(t,r)$ starting blocks required to construct a 3D grid graph that completely covers $G_{m,n,k}$, so that every edge and vertex in $G_{m,n,k}$ belongs to at least one $(t,r)$ starting block. Since each starting block is efficiently $(t,r)$ dominated by $2$ towers, then $\gamma_{t,r} (G_{m,n,k}) \leq 2 \cdot B$.
\end{restatable}

Finally, we consider the slant grid graph with $m$ rows and $n$ columns, denoted $S_{m,n}$, and the king's grid graph with $m$ rows and $n$ columns, denoted $K_{m,n}$. We provide an explicit formula for the $(t,r)$ broadcast domination number of $S_{2,n}$ when $t > r$.

\begin{restatable}[]{thm}{slantdom}
\label{2xn_slant_dom}
If $t>r$, then $\gamma_{t,r} (S_{2,n}) = \left \lceil \frac{2(n+r-1)}{4t-2r-1} \right \rceil$. 
\end{restatable}

We also establish an explicit formula for the domination number of $K_{m,n}$ when $t > r$ and $m \leq 2(t-r) + 1$.  

\begin{restatable}[]{thm}{kingmxn}
\label{thm:kingmxn}
For $t > r$ and $ m \leq 2(t-r) + 1$, then 
$\gamma_{t,r}(K_{m,n})= \left \lceil \frac{n+r-1}{2t-r} \right \rceil.$
\end{restatable}

Furthermore, we find efficient broadcasting domination patterns, a placement of broadcasting towers such that each vertex receives at least $r$ signal, on the infinite king's lattice for $(t,1)$ and $(t,2)$.

\subsection*{Outline of Paper} In Section \ref{PathSection}, we give the $(t,r)$ broadcast domination number of paths on $n$ vertices in terms of $t, r,$ and $n$. Section \ref{GridGraphSection} is focused on proving the upper bound on the $(t,r)$ broadcast domination number of $m \times n$ grid graphs and develops the notion of a $(t,r)$ \textbf{starting block} which is used throughout the remainder of the paper. We then generalize our results on the $m \times n$ grid graph to the 3D grid graph in Section \ref{3DGridSection} and give upper bounds on the $(t,r)$ broadcast domination number of $m \times n \times k$ $3D$ grid graphs. Finally, in Section \ref{LatticeSections} we give efficient broadcast domination sets for certain $t$ and $r$ for both the finite king's grid graph and the infinite king's lattice. We also leverage previous results of Harris--Luque--Flores--Sepulveda \cite{infinite_triangle} for the infinite triangular lattice to give upper bounds on the $(t,r)$ broadcast domination number of the $m \times n$ slant grid graph.

\subsection*{Acknowledgements}
This research was supported in part by the Alfred P. Sloan Foundation, the Mathematical Sciences Research Institute, and the National Science Foundation (grant No. DMS-1156499). We would like to thank Rebecca Garcia for her feedback on earlier drafts of this manuscript.


\section{\texorpdfstring{$(t,r)$}{} Broadcast Domination of Paths}
\label{PathSection}
In this section we study the $(t,r)$ broadcast domination number of paths on $n$ vertices. We find the $(t,r)$ broadcast domination number of $P_n$ for any $t \geq r \geq 1$. We begin by considering the size of the overlap needed between two towers in order to build an efficient broadcast for $P_n$.

\begin{lemma} \label{vertex_overlap}
Given a $(t, r)$ broadcast dominating set $D$ of a path $P_n$, any two consecutive broadcast vertices in $D$ have an overlap in their broadcast zones of at least $r - 1$ vertices, with the overlap being precisely $r-1$ vertices when two consecutive broadcast vertices are distance $2t-r$ from each other, so that vertices in the overlap receive exactly $r$ amount of signal.
\end{lemma}

\begin{proof}
Consider $\gamma_{t, t-k}(P_n)$, which is the size of an optimal $(t, t-k)$ broadcast dominating set. Without overlap and when placing towers left to right, there will be $t-k-1$ vertices to the right of the tower that receive insufficient signal. Each of these vertices needs to receive signal from another tower. Thus, the amount of overlap between two broadcast towers needs to be at least $t-k-1$. Note that the required reception is just $r = t - k$, so the amount of the overlapped vertices is at least $t - k - 1 = r- 1,$ as desired.
\end{proof}

We are now ready to prove Theorem \ref{thm:trpath} using Lemma \ref{vertex_overlap}.

\trpath*
\begin{proof}

\begin{figure}[htb!]
    \centering
\begin{tikzpicture}
\draw[thick](0,0)--(2,0);
\draw[thick] (3,0)--(6,0);
\draw[thick] (7,0)--(10,0);
\draw[thick] (11,0)--(13,0);

\fill[black] (2.15,0) circle (1pt);
\fill[black] (2.5,0) circle (1pt);
\fill[black] (2.85,0) circle (1pt);

\fill[black] (6.15,0) circle (1pt);
\fill[black] (6.5,0) circle (1pt);
\fill[black] (6.85,0) circle (1pt);

\fill[black] (0,0) circle (3pt);
\node at (0,-.5) {$r$};
\fill[black] (1,0) circle (3pt);
\node at (1,-.5) {$r+1$};

\fill[black] (4,0) circle (3pt);
\node at (4,-.5) {$t$};
\fill[black] (5,0) circle (3pt);
\node at (5,-.5) {$t-1$};

\fill[black] (8,0) circle (3pt);
\node at (8,-.5) {$r$};
\fill[black] (9,0) circle (3pt);
\node at (9,-.5) {$r-1$};

\fill[black] (10.15,0) circle (1pt);
\fill[black] (10.5,0) circle (1pt);
\fill[black] (10.85,0) circle (1pt);

\fill[black] (12,0) circle (3pt);
\node at (12,-.5) {$2$};
\fill[black] (13,0) circle (3pt);
\node at (13,-.5) {$1$};
\end{tikzpicture}
\caption{First Vertices of the Efficiently Dominated Path}
    \label{fig:pathconstructionex}
\end{figure}
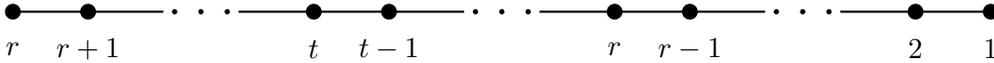

To construct an optimal dominating set by maximizing the size of a broadcast tower's broadcast zone, we want the first vertex of the path to have reception $r$, as shown in Figure \ref{fig:pathconstructionex}. There are $t-r$ vertices from the first vertex to the broadcast tower and $t$ vertices following the broadcast tower that receive reception from this tower, so the size of the broadcast zone is then $2t - r$. So, we want to partition the path into groups of $2t - r$ vertices, and place a tower in each partition. If no overlap was necessary, the number of towers used would be$ \left \lceil \frac{n}{2t - r} \right  \rceil$. However, there is necessary overlap required so that all vertices have at least reception $r$. We  know $r-1$ vertices in each broadcast zone also belong to another broadcast zone. If overlap is required, which is when $r > 1$, we need more towers. To account for the towers needed for the $r-1$ overlapped vertices in each partition, which are size $2t-r$, we must add $r-1$ to the number of vertices $n$. Therefore, we've constructed  a dominating set on a path with $n$ vertices that uses $\left \lceil \frac{n+r-1}{2t-r} \right \rceil$ towers.

Now we will show that the dominating set of order $\left \lceil \frac{n+(r-1)}{2t - r} \right \rceil$ is of minimum cardinality. Let $P$ be a path on $n$ vertices. We know $\left \lceil \frac{n+(r-1)}{2t-r} \right \rceil $ is a dominating set. Assume there is a smaller dominating set of order $\left \lceil \frac{n+(r-1)}{2t-r} \right \rceil -1$. However, since by assumption this is a $(t,r)$ dominating set, there must be an overlap of at least $r-1$ vertices in the broadcast zones of two consecutive broadcast towers by Lemma \ref{vertex_overlap}. Therefore, the number of vertices between two consecutive broadcasting towers is at most $(t-r) + (t-r) + r-1 = 2t - r -1$. Then, the distance between two consecutive broadcasting towers is required to be at most $2t-r$. If we remove one tower, we lose one broadcast zone. This means that there are vertices with insufficient reception. We can try to shift the remaining towers to cover those vertices, but the requirement for the distance between consecutive broadcasting towers to be at most $2t - r$ would leave us with insufficient reception at the beginning or end of the path. So, there exists no smaller dominating set.
\end{proof}

We wrote an algorithm (in SageMath) which generates a dominating set for $P_n$ for any given $t, r,$ and $n$. See \cref{ap:algorithm}.



\section{\texorpdfstring{$(t,r)$}{} Broadcast Domination of Grid Graphs}
\label{GridGraphSection}

Significant work on the $(t, r)$ broadcast domination number of grid graphs was carried out by Blessing--Insko--Johnson--Mauretour in \cite{gridgraph}, as outlined in the introduction. We now extend these results in several directions, both for more general values of $t$ and $r$ and for larger families of grid graphs. 

To generate our upper bound of the $(t,r)$ broadcast domination number, we introduce the notion of a $(t,r)$ {\bf starting block}, which is a grid that is efficiently dominated by two towers of strength $t$.

\begin{theorem} \label{condensed_block_results}\label{starting_block_2xn} \label{mxn_starting_block} \label{2_2_k_block_r=1}
Let $G_{m_1,...,m_d}$ be a $d$-dimensional grid graph and $(t,r)$ be such that $t \geq d$ and $t\geq r \geq 1$. If $m_1+m_2+\cdots +m_d = 2t-r+d$, then $G_{m_1,m_2,\ldots, m_d}$ is efficiently dominated by $2$ broadcasting towers placed on antipodal vertices of $G_{m_1,m_2, \ldots, m_d}$, such as the vertices $(1,1,\ldots, 1)$ and $(m_1,m_2,\ldots, m_d)$.
\end{theorem}
\begin{proof}
Assume the two broadcasting vertices are placed at antipodal vertices of $G$. Since $m_1+m_2+\cdots +m_d = 2t-r+d$, every vertex in the grid $G_{m_1,m_2,\ldots, m_d}$ lies on a path of length $2t - r$ between the vertices $(1,1,\ldots,1)$ and $(m_1,m_2,\ldots, m_d)$. An example of such a path in $G_{m,n,k}$ is highlighted in Figure \ref{fig:path_btw_towers}. By Lemma \ref{vertex_overlap}, such a path is efficiently dominated, and therefore the entire grid graph is efficiently dominated by the two towers of strength $t$.
\end{proof}

\begin{figure}[htb!]
    \tdplotsetmaincoords{70}{110}

    \begin{tikzpicture}[tdplot_main_coords]
    \newcommand{\mm}{5}
    \newcommand{\nn}{4}
    \newcommand{\kk}{3}

    \draw[-] (0,0,0) -- (\mm,0,0) node[anchor=north east]{$m$};
    \draw[-] (0,0,0) -- (0,\nn,0) node[anchor=north west]{$n$};
    \draw[-] (0,0,0) -- (0,0,\kk);
    \node[black] at (0,-.5,\kk) {$k$};


    \draw[thick, black] (\mm, 0, \kk) to (\mm, \nn, \kk);
    \draw[thick, black] (0, \nn, \kk) to (0, \nn, 0);
    \draw[thick, black] (\mm, 0, 0) to (\mm, 0, \kk) to (0, 0, \kk); 
    \draw[thick, black] (\mm, 0, 0) to (\mm, \nn, 0) to (0, \nn, 0); 

    \draw[thin, black] (0, \nn, 0) to (0,0,0) to (\mm, 0, 0);

    \draw[thin, black] (0, 0, \kk) to (0, 0 ,0);

    \draw[thin, black] (\mm/2, 0, \kk) to (\mm/2, \nn, \kk) to (\mm/2, \nn, 0)  to (\mm/2, 0, 0) to cycle;
    
    \draw[thin, black] (0, \nn/2, \kk) to (\mm, \nn/2, \kk) to (\mm, \nn/2, 0) to (0, \nn/2, 0) to cycle;
    
    \draw[thin, black] (\mm, 0, \kk/2) to (\mm, \nn, \kk/2) to (0, \nn, \kk/2) to (0, 0, \kk/2) to cycle;
    
    \draw[thin, black] (\mm/2, \nn/2, \kk) to (\mm/2, \nn/2, 0);
    \draw[thin, black] (\mm/2, 0, \kk/2) to (\mm/2, \nn, \kk/2);
    \draw[thin, black] (0, \nn/2, \kk/2) to (\mm, \nn/2, \kk/2);

    \fill[gray, nearly transparent] (\mm, 0, 0) to ( \mm, 0, \kk) to (\mm, \nn, \kk) to (\mm, \nn, 0) to cycle;
    \fill[gray, opacity=0.5] (\mm, \nn, 0) to (0, \nn, 0) to (0, \nn, \kk) to (\mm, \nn, \kk) to cycle;
    \fill[gray, opacity=0.7] (0, 0, \kk) to (\mm, 0, \kk) to (\mm, \nn, \kk) to (0, \nn, \kk) to cycle;
    
    \draw[line width=2pt, green] (0,0,\kk) to (0,\nn,\kk) to (\mm,\nn,\kk) to (\mm,\nn,0);
    
    \fill[black] (\mm,0,0) circle (2 pt);
    \fill[black] (\mm/2,\nn,0) circle (2 pt);
    \fill[black] (0,\nn,0) circle (2 pt);
    \fill[black] (\mm, \nn/2, 0) circle (2pt);
    \fill[black] (\mm,0,\kk/2) circle (2 pt);
    \fill[black] (\mm,\nn/2,\kk/2) circle (2 pt);
    \fill[black] (\mm,\nn,\kk/2) circle (2 pt);
    \fill[black] (\mm/2,\nn,\kk/2) circle (2 pt);
    \fill[black] (0,\nn,\kk/2) circle (2 pt);
    \fill[black] (0,\nn,\kk) circle (2 pt);
    \fill[black] (\mm/2,\nn,\kk) circle (2 pt);
    \fill[black] (\mm,\nn,\kk) circle (2 pt);
    \fill[black] (\mm,\nn/2,\kk) circle (2 pt);
    \fill[black] (\mm,0,\kk) circle (2 pt);
    \fill[black] (\mm/2,0,\kk) circle (2 pt);

    \fill[black] (0,\nn/2,\kk) circle (2 pt);    
    \fill[black] (\mm/2,\nn/2,\kk) circle (2 pt);   
    \fill[black] (\mm/2, 0, \kk/2) circle (2 pt);
    \fill[black] (\mm/2, \nn/2, 0) circle (2 pt);
    \fill[black] (\mm/2, 0, 0) circle (2 pt);
    \fill[black] (0, \nn/2, 0) circle (2 pt);
    \fill[black] (0, \nn/2, \kk/2) circle (2 pt);
    \fill[black] (\mm/2, \nn/2, \kk/2) circle (2 pt);
    \fill[black] (0, 0, \kk/2) circle (2 pt);
    \fill[black] (0,0,0) circle (2 pt);
    
    \draw[dashed, black] (0,0,\kk) to (0,\nn,\kk) to (\mm,\nn,\kk) to (\mm, \nn, 0);
    \fill[red] (0,0,\kk) circle (2 pt);
    \fill[red] (\mm,\nn,0) circle (2 pt);
   
    \end{tikzpicture}
    \caption{Shortest path between towers at positions $(1,1,k)$ and $(m,n,1)$.}
    \label{fig:path_btw_towers}
\end{figure}

We illustrate Theorem \ref{condensed_block_results} for $d=2$ in the following example. First, consider $G_{2,3}$ with $(t,r) = (2,1)$. We have $2 + 3 = 2(2) - 1 + 2$, so $G_{2,3}$ is efficiently dominated by two towers of strength $2$, as shown in Figure \ref{fig:2x3_ex}.

\begin{figure}[htb!]
\centering
    \begin{tikzpicture}[scale=2]
    \foreach \i in {0,1}{
    \foreach \j in {0}{
    
    \draw[thick,black] (\i,\j)--(\i,\j+1);
    \draw[thick,black] (\i,\j)--(\i+1,\j);
    \fill[black] (\i,\j) circle (3pt);
    \fill[black] (\i,\j+1) circle (3pt); 
    \fill[black] (\i+1,\j) circle (3pt);
    \fill[black] (\i+1,\j+1) circle (3pt); 
    }
    }
    \draw[thick, black] (0,1)--(2,1)--(2,0);
    \fill[black] (0,1) circle (3pt);
    \fill[black] (2,1) circle (3pt);

    \fill[red] (0,1) circle (3pt);
    \fill[red] (2,0) circle (3pt);
    \node[red] at (0,1.25) {2};
    \node[red] at (2, -.25) {2};
    
    \end{tikzpicture}
    \caption{The $2 \times 3$ grid, $(2,1)$ dominated by $2$ towers.}
    \label{fig:2x3_ex}
\end{figure}

We now use Theorem \ref{mxn_starting_block} to give our general results for $\gamma_{t,r}(G_{m,n})$.

\mxntr*


\begin{proof}
We will construct our $(t, r)$ broadcast domination set beginning with a starting block with two broadcast vertices. From Theorem \ref{mxn_starting_block}, for $m, t,r$, the dimension of the starting block is the grid $G_{m,n'}$, where $n' = 2t-r-(m-2)$. We now need to dominate the remaining $n - (2t-r-(m-2))$ columns of $G_{m,n}$. We will place the next tower after the starting block $2t-r-(m-2)-1$ columns away in the top row. We continue to distribute towers at these distances, alternating their placement between the top and bottom row until we reach the end of the graph. Adding the starting block's two towers, we arrive at the formula $2 + \left \lceil \frac{n- (2t - r - (m-2))}{(2t - r - (m-2))-1} \right \rceil$, which is the size of our dominating set. \end{proof}

\section{\texorpdfstring{$(t,r)$}{} Broadcast Domination of the 3D Grid Graph}
\label{3DGridSection}

In this section, we consider $3$-dimensional grid graphs of size $m \times n \times k$, whose vertices are subsets of $\mathbb Z^3$ with edges joining vertices that are distance $1$ apart. We will denote these 3D grid graphs by $G_{m, n, k}$ where $m, n, k \in \mathbb N$. We begin by considering a simple case, namely the $(2,1)$ broadcast domination number for $G_{2, 2, k}$ when $k \geq 1$.


\begin{theorem}
If $k > 1$, then $\gamma_{2,1}(G_{2,2, k}) = k$. 
\end{theorem}

\begin{proof}
Figure \ref{fig:2x2x2_(2,1)dom} illustrates how $G_{2,2,5}$ can be $(2,1)$ broadcast dominated by a set $D \subset V(G)$, in the case that $k = 5$. We can construct a $(2,1)$ broadcast dominating set $D$ for any $k$ by continuing this pattern of placing a tower at each level in the $k$ direction and alternating the placement of these towers between pairs of vertices on opposite diagonals in the $(m,n)$ plane. Notice that at every level in the $k$ direction the two vertices in the $(m,n)$ plane that are adjacent to a broadcast vertex only receive signal $1$. So $D$ is indeed a minimum $(2,1)$ broadcast dominating set; in fact, $D$ is a minimum dominating set.
\end{proof}

\begin{figure}[htb!]
    \tdplotsetmaincoords{70}{110}

    \begin{tikzpicture}[tdplot_main_coords]
    \newcommand{\mm}{3}
    \newcommand{\nn}{3}
    \newcommand{\kk}{8}

    \draw[thick] (0,0,0) -- (\mm,0,0) node[anchor=north east]{$m = 2$};
    \draw[thick] (0,0,0) -- (0,\nn,0) node[anchor=north west]{$n = 2$};
    \draw[thick] (0,0,0) -- (0,0,\kk) node at(0,0,8.5){$k = 5$};

    
    \draw[thick, black] (0,0,\kk) to (0,\nn,\kk) to (\mm,\nn,\kk) to (\mm,\nn,0);
    \draw[thick, black] (\mm, 0, \kk) to (\mm, \nn, \kk);
    \draw[thick, black] (0, \nn, \kk) to (0, \nn, 0);
    \draw[thick, black] (\mm, 0, 0) to (\mm, 0, \kk) to (0, 0, \kk); 
    \draw[thick, black] (\mm, 0, 0) to (\mm, \nn, 0) to (0, \nn, 0); 
    \draw[thick, black] (0,0, \kk -6) to (\mm,0,\kk -6);
    \draw[thick, black] (\mm,\nn,\kk -6) to (\mm,0,\kk -6);
    \draw[thick, black] (\mm,\nn,\kk -6) to (0,\nn,\kk -6);
    \draw[thick,black]  (0,0, \kk -6) to (0,\nn,\kk -6);
    
     \draw[thick, black] (0,0, \kk -4) to (\mm,0,\kk -4);
    \draw[thick, black] (\mm,\nn,\kk -4) to (\mm,0,\kk -4);
    \draw[thick, black] (\mm,\nn,\kk -4) to (0,\nn,\kk -4);
    \draw[thick,black]  (0,0, \kk -4) to (0,\nn,\kk -4);
     \draw[thick, black] (0,0, \kk -2) to (\mm,0,\kk -2);
    \draw[thick, black] (\mm,\nn,\kk -2) to (\mm,0,\kk -2);
    \draw[thick, black] (\mm,\nn,\kk -2) to (0,\nn,\kk -2);
    \draw[thick,black]  (0,0, \kk -2) to (0,\nn,\kk -2);
     \draw[thick, black] (0,0, \kk) to (\mm,0,\kk);
    \draw[thick, black] (\mm,\nn,\kk) to (\mm,0,\kk);
    \draw[thick, black] (\mm,\nn,\kk) to (0,\nn,\kk);
    \draw[thick,black]  (0,0, \kk) to (0,\nn,\kk);

    \fill[gray, nearly transparent] (\mm, 0, 0) to ( \mm, 0, \kk) to (\mm, \nn, \kk) to (\mm, \nn, 0) to cycle;
    \fill[gray, opacity=0.5] (\mm, \nn, 0) to (0, \nn, 0) to (0, \nn, \kk) to (\mm, \nn, \kk) to cycle;
    \fill[gray, opacity=0.7] (0, 0, \kk) to (\mm, 0, \kk) to (\mm, \nn, \kk) to (0, \nn, \kk) to cycle;

    \fill[black](0,0,0) circle (2pt);
    \fill[red](\mm,0,0) circle (4pt);
    \fill[black](0,\nn,0) circle (2pt);
    \fill[black](\mm,\nn,0) circle (2pt);
    \fill[black](0,0, \kk -6) circle (2pt);
    \fill[black](\mm,0,\kk -6) circle (2pt);
    \fill[black](\mm,\nn,\kk -6) circle (2pt);
    \fill[red](0,\nn,\kk -6) circle (4pt);
    \fill[black](0,0,\kk -4) circle (2pt);
    \fill[red](\mm,0,\kk -4) circle (4pt);
    \fill[black](\mm,\nn,\kk -4) circle (2pt);
    \fill[black](0,\nn, \kk -4) circle (2pt);
    
    \fill[black](0,0,\kk -2) circle (2pt);
    \fill[black](\mm,0,\kk -2) circle (2pt);
    \fill[red](0,\nn,\kk -2) circle (4pt);
    \fill[black](\mm,\nn,\kk -2) circle (2pt);
    \fill[black](0,0, \kk) circle (2pt);
    \fill[red](\mm,0,\kk) circle (4pt);
    \fill[black](\mm,\nn,\kk) circle (2pt);
    \fill[black](0,\mm,\kk) circle (2pt);

    \end{tikzpicture}
    \caption{A $(2,1)$ broadcast domination set for $G_{2,2,5}$ with broadcast towers in red.}
   
    \label{fig:2x2x2_(2,1)dom}
\end{figure}

We now want to consider the $(t,r)$ {\bf starting block} from Theorem \ref{condensed_block_results} for $d = 3$. Using Theorem \ref{2_2_k_block_r=1}, we construct a $(t,r)$ starting block of the 3D grid graph for $r \geq 1$.

\begin{lemma}
The $(t,r)$ starting block for $G_{2,2,k}$ is a $2 \times 2 \times (2t - r - 1)$ grid when $k \geq 2t - r -1$.
\end{lemma}

\begin{proof}
We will proceed by induction on $r$. From Theorem \ref{2_2_k_block_r=1}, we have $2 + 2 + 2t - 1 - 1 = 2t+2 = 2t - 1 + 3$, so the $2 \times 2 \times (2t - r- 1)$ grid graph is a $(t,1)$ starting block. Now assume $\gamma_{t,r}(G_{2,2, (2t -r - 1)}) = 2$. We will show $\gamma_{t,r}(G_{2,2, (2t -r - 2)}) = 2$. If we increase $r$ to $r+1$, then vertices that received a signal of $r$ will no longer be dominated by our pattern of towers. Therefore, we must move one of the towers distance one closer to the other in our starting block. Specifically, we choose to move a tower one distance closer in the $k$ direction. This reduces the number of layers in our starting block by $1$ from $k=2t-r-1$ to $k=2t-r-2$.
\end{proof}

Our main results for $3$D grid graphs is the following upper bound for all $G_{m,n,k}$, developed using the concept of starting blocks.
\upperbound*

\begin{proof}
We've defined a starting block to be a grid graph that is efficiently dominated by two towers. Consider some 3D graph $G_{m,n,k}$ and cover it using starting blocks of a certain size. We then have $(t,r)$ dominated $G_{m,n,k}$, since each starting block is dominated. Therefore, we have constructed a dominating set. There are $2$ towers per each starting block, so we know the size of the dominating set, which is $2 \cdot B$, where $B$ is the number of starting blocks used. Since we've constructed a dominating set, we have that \[\gamma_{t,r} (G_{m,n,k}) \leq 2 \cdot B\] as desired.
\end{proof}

\section{\texorpdfstring{$(t,r)$}{} Broadcast Domination of Other Grid Graphs}
\label{LatticeSections}
In this section we study the $(t,r)$ broadcast domination number of slant grid graphs and the king's grid graphs. Through this section we give efficient $(t,r)$ broadcast domination sets for certain values of $t$ and $r$ for both the king's grid graph and the infinite king's lattice. We also leverage previous results of Harris--Luque--Flores--Sepulveda \cite{infinite_triangle} for the infinite triangular lattice to give upper bounds on the $(t,r)$ broadcast domination number of the $m \times n$ slant grid graph.

\subsection{Slant grid graphs}
A slant lattice is a graph similar to the grid graph, but each vertex is adjacent to a vertex diagonal from its position. There are two options for the direction of the slant, we will choose the one pictured in Figure \ref{fig:slant_lattice_ex}.

\begin{figure}[htb!]
    \centering
    \begin{tikzpicture}[scale=0.75]
    \foreach \i in {0,1,2,...,5}{
        \foreach \j in {0,1,2,...,6}{

        \draw[thick,black] (\i,\j)--(\i+1,\j+1);
        \draw[thick,black] (\i,\j)--(\i,\j+1);
        \draw[thick,black] (\i,\j)--(\i+1,\j);
        \fill[black] (\i,\j) circle (3pt);
        }
    }
    \draw[thick, black] (0,7)--(1,7)--(2,7)--(3,7)--(4,7)--(5,7)--(6,7)--(6,6)--(6,5)--(6,4)--(6,3)--(6,2)--(6,1)--(6,0);
    
    \fill[black] (0,7) circle(3pt);
    \fill[black] (1,7) circle(3pt);
    \fill[black] (2,7) circle(3pt);
    \fill[black] (3,7) circle(3pt);
    \fill[black] (4,7) circle(3pt);
    \fill[black] (5,7) circle(3pt);
    \fill[black] (6,7) circle(3pt);
    \fill[black] (6,6) circle(3pt);
    \fill[black] (6,5) circle(3pt);
    \fill[black] (6,4) circle(3pt);
    \fill[black] (6,3) circle(3pt);
    \fill[black] (6,2) circle(3pt);
    \fill[black] (6,1) circle(3pt);
    \fill[black] (6,0) circle(3pt);
    \end{tikzpicture}
    \caption{Slant grid graph of $S_{8,7}$.}
    \label{fig:slant_lattice_ex}
\end{figure}
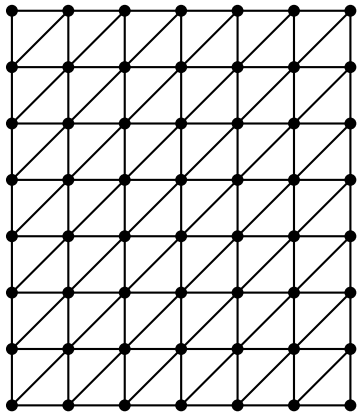
We will let $S_{m,n}$ denote a slant grid graph on $m$ rows and $n$ columns.
We note that the infinite slant lattice is isomorphic to the triangular lattice, and Harris--Luque--Flores--Sepulveda extensively studied the efficient domination of the infinite triangular lattice in \cite{infinite_triangle}. Their paper also examines the $(t,r)$ broadcast domination of the triangular matchstick graph, which is a specific type of subgraph of the infinite triangular lattice. Note that $S_{m,n}$ is also a subgraph of the triangular lattice. We will apply one of the main results of \cite{infinite_triangle}, which we include here as Theorem \ref{harris_coord_thm}, to generate an upper bound for the $(t,r)$ broadcast domination number of $S_{m,n}$.

\begin{theorem}[Theorem 4.1, \cite{infinite_triangle}] \label{harris_coord_thm}
Let $t \geq r \geq 1$. Then an efficient $(t,r)$ broadcast domination pattern for the infinite triangular grid is given by placing a tower at every vertex of the form
\[[(2t-r)x + (t-r)y]\alpha_1 + [tx + (2t-r)y]\alpha_2 \]
with $x,y \in \ZZ$ and $\alpha_1 = (1,0)$ and $\alpha_2 = (-\frac{1}{2}, \frac{\sqrt{3}}{2})$.
\end{theorem}

Theorem \ref{harris_coord_thm} was stated for the infinite triangular grid, where the interior triangles are equilateral. However, if we let $\alpha_1=(-1,0)$ and $\alpha_2=(1,1)$, then Theorem \ref{harris_coord_thm} describes the placement of towers on $S_{m,n}$. Since we are thinking of $S_{m,n}$ as a subgraph of the infinite triangular lattice, we will always assume that its bottom left corner is at the origin. We will also assume that this vertex is always a dominating vertex, as outlined in \cite{infinite_triangle}. Note that a tower in $S_{m,n}$ has as a hexagonal broadcast zone.

\begin{lemma} \label{start_tile_row}
A single tower at the origin can dominate $S_{n,n}$, if $n= t-r+1$. 
\end{lemma}

\begin{proof}
The distance from the origin, which is a broadcasting tower, to the farthest vertex that can be $(t,r)$ dominated by that tower is $t-r$. That implies a grid that is $t-r+1$ by $t-r+1$ vertices is dominated by a single tower at the origin.
\end{proof}

In order to construct $(t,r)$ broadcast dominating sets for the $m \times n$ slant grid graph we introduce the notion of a \textbf{starting tile}. A $(t,r)$ starting tile is an $m \times n$ slant grid graph, where $m = t-r+1$ and $n$ is the number of columns required so that the last tower dominating the tile is at $(p,0)$, for some $p \in \mathbb N$. See Figure \ref{fig:(2,1)_slant_example} for an example.

When looking for a good starting tile, Lemma \ref{start_tile_row} dictates the number of rows of the tile. Then, using Theorem \ref{harris_coord_thm}, we can generate upper bounds for $\gamma_{t,r} (S_{m,n})$ for certain values of $(t,r)$.

\begin{figure}[htb!]
    \centering
    \tikzset{every picture/.style={line width=0.75pt}} 

\begin{tikzpicture}[x=0.75pt,y=0.75pt,yscale=-1,xscale=1]

\draw  [fill={rgb, 255:red, 137; green, 233; blue, 134 }  ,fill opacity=0.47 ] (130.5,100) -- (169.5,100) -- (169.5,100) -- (169.5,130.5) -- (139,161) -- (100,161) -- (100,161) -- (100,130.5) -- cycle ;
\draw  [fill={rgb, 255:red, 0; green, 0; blue, 0 }  ,fill opacity=1 ] (131.75,130.5) .. controls (131.75,128.84) and (133.09,127.5) .. (134.75,127.5) .. controls (136.41,127.5) and (137.75,128.84) .. (137.75,130.5) .. controls (137.75,132.16) and (136.41,133.5) .. (134.75,133.5) .. controls (133.09,133.5) and (131.75,132.16) .. (131.75,130.5) -- cycle ;
\draw  [fill={rgb, 255:red, 137; green, 233; blue, 134 }  ,fill opacity=0.47 ] (221.5,70) -- (260.5,70) -- (260.5,70) -- (260.5,100.5) -- (230,131) -- (191,131) -- (191,131) -- (191,100.5) -- cycle ;
\draw  [fill={rgb, 255:red, 0; green, 0; blue, 0 }  ,fill opacity=1 ] (222.75,100.5) .. controls (222.75,98.84) and (224.09,97.5) .. (225.75,97.5) .. controls (227.41,97.5) and (228.75,98.84) .. (228.75,100.5) .. controls (228.75,102.16) and (227.41,103.5) .. (225.75,103.5) .. controls (224.09,103.5) and (222.75,102.16) .. (222.75,100.5) -- cycle ;
\draw  [fill={rgb, 255:red, 137; green, 233; blue, 134 }  ,fill opacity=0.47 ] (261.5,131) -- (300.5,131) -- (300.5,131) -- (300.5,161.5) -- (270,192) -- (231,192) -- (231,192) -- (231,161.5) -- cycle ;
\draw  [fill={rgb, 255:red, 0; green, 0; blue, 0 }  ,fill opacity=1 ] (262.75,161.5) .. controls (262.75,159.84) and (264.09,158.5) .. (265.75,158.5) .. controls (267.41,158.5) and (268.75,159.84) .. (268.75,161.5) .. controls (268.75,163.16) and (267.41,164.5) .. (265.75,164.5) .. controls (264.09,164.5) and (262.75,163.16) .. (262.75,161.5) -- cycle ;
\draw  [fill={rgb, 255:red, 137; green, 233; blue, 134 }  ,fill opacity=0.47 ] (311.5,39) -- (350.5,39) -- (350.5,39) -- (350.5,69.5) -- (320,100) -- (281,100) -- (281,100) -- (281,69.5) -- cycle ;
\draw  [fill={rgb, 255:red, 0; green, 0; blue, 0 }  ,fill opacity=1 ] (312.75,69.5) .. controls (312.75,67.84) and (314.09,66.5) .. (315.75,66.5) .. controls (317.41,66.5) and (318.75,67.84) .. (318.75,69.5) .. controls (318.75,71.16) and (317.41,72.5) .. (315.75,72.5) .. controls (314.09,72.5) and (312.75,71.16) .. (312.75,69.5) -- cycle ;
\draw  [fill={rgb, 255:red, 137; green, 233; blue, 134 }  ,fill opacity=0.47 ] (351.5,100) -- (390.5,100) -- (390.5,100) -- (390.5,130.5) -- (360,161) -- (321,161) -- (321,161) -- (321,130.5) -- cycle ;
\draw  [fill={rgb, 255:red, 0; green, 0; blue, 0 }  ,fill opacity=1 ] (352.75,130.5) .. controls (352.75,128.84) and (354.09,127.5) .. (355.75,127.5) .. controls (357.41,127.5) and (358.75,128.84) .. (358.75,130.5) .. controls (358.75,132.16) and (357.41,133.5) .. (355.75,133.5) .. controls (354.09,133.5) and (352.75,132.16) .. (352.75,130.5) -- cycle ;
\draw  [fill={rgb, 255:red, 137; green, 233; blue, 134 }  ,fill opacity=0.47 ] (170.5,160) -- (209.5,160) -- (209.5,160) -- (209.5,190.5) -- (179,221) -- (140,221) -- (140,221) -- (140,190.5) -- cycle ;
\draw  [fill={rgb, 255:red, 0; green, 0; blue, 0 }  ,fill opacity=1 ] (171.75,190.5) .. controls (171.75,188.84) and (173.09,187.5) .. (174.75,187.5) .. controls (176.41,187.5) and (177.75,188.84) .. (177.75,190.5) .. controls (177.75,192.16) and (176.41,193.5) .. (174.75,193.5) .. controls (173.09,193.5) and (171.75,192.16) .. (171.75,190.5) -- cycle ;

\end{tikzpicture}
    \caption{An efficient $(2,1)$ dominating pattern for the slant lattice.}
    \label{fig:slantpatternexample}
\end{figure}
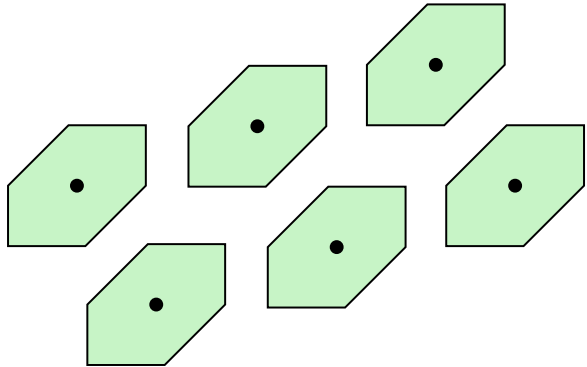
If we take a slice of the slant lattice with an efficient dominating pattern, as depicted in Figure \ref{fig:slantpatternexample} such that the bottom corners of the slice are towers, we get a starting tile, which is depicted in Figure \ref{fig:(2,1)_slant_example}. The red edges depict the broadcast zone of five towers, three of which are on the tile, and the red vertices depict the towers.
\begin{figure}[htb!]
    \centering
    \begin{tikzpicture}[scale=.5]
    \draw[semithick, black] (0,0)--(4,0)--(6,2)--(8,2);
    \draw[semithick, black] (0,2)--(0,0);
    \draw[semithick, black] (2,2)--(4,2)--(6,2);
    \draw[semithick, black] (0,0)--(2,2);
    \draw[semithick, black] (6,0)--(6,2);
    \draw[semithick, black] (8,0)--(8,2)--(10,2);
    \draw[semithick, black] (12,2)--(14,2)--(14,0)--(10,0)--(10,2);
    \draw[semithick, black] (12,0)--(12,2);
    
    \draw[semithick, black] (2,0)--(4,2);

    \draw[semithick, black] (8,0)--(10,2);
    \draw[semithick, black] (10,0)--(12,2);
    \draw[semithick, black] (6,0)--(8,0);

    \draw[ultra thick, green] (8,0)--(10,0);
    \draw[very thin ,black, dashed] (8,0)--(10,0);
    
    \draw[ultra thick, green] (10,2)--(12,2);
    \draw[very thin ,black, dashed] (10,2)--(12,2);
    \draw[ultra thick, green] (2,2)--(0,2);
    \draw[very thin ,black, dashed] (2,2)--(0,2);
    \draw[ultra thick, green] (4,2)--(4,0)--(6,0)--(8,2);
    \draw[very thin ,black, dashed] (4,2)--(4,0)--(6,0)--(8,2);
    \draw[ultra thick, green] (2,0)--(2,2);
    \draw[very thin ,black, dashed] (2,0)--(2,2);

    \draw[ultra thick, green] (12,0)--(14,2);
    \draw[very thin ,black, dashed] (12,0)--(14,2);

    \fill[green] (0,0) circle (5pt);
    \fill[black] (2,0) circle (3pt);
    \fill[black] (4,0) circle (3pt);
    \fill[black] (6,0) circle (3pt);
    \fill[black] (8,0) circle (3pt);
    \fill[black] (10,0) circle (3pt);
    \fill[black] (12,0) circle (3pt);
    \fill[green] (14,0) circle (5pt);

    \fill[black] (0,2) circle (3pt);
    \fill[black] (2,2) circle (3pt);
    \fill[black] (4,2) circle (3pt);
    \fill[green] (6,2) circle (5pt);
    \fill[black] (8,2) circle (3pt);
    \fill[black] (10,2) circle (3pt);
    \fill[black] (12,2) circle (3pt);
    \fill[black] (14,2) circle (3pt);

    \end{tikzpicture}
    \caption{A $(2,1)$ starting tile of the slant grid graph.}
    \label{fig:(2,1)_slant_example}
\end{figure}
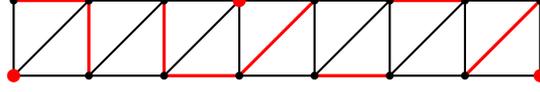

\begin{theorem} \label{2,1_slant_tiling}
For $S_{m,n}$, if $m= 2p$ and $n=8q$, where $p,q \in \mathbb{N}$, then \[\gamma_{2,1} (S_{m,n}) \leq (4q+1)p. \]
If $m = 2p + \ell$ for $\ell \in \{0,1\}$ and $n = 8q + k$ for $k \in \{1, 2, ...,7\}$, then \[\gamma_{2,1} (S_{m,n}) \leq (4(q+1)+1)(p+1) \]
if either $\ell$ and $k$ are nonzero.
\end{theorem}

\begin{proof}
Taking a $2\times 8$ slice of the the efficient domination pattern from Theorem \ref{harris_coord_thm} on the infinite lattice, we generate a $2 \times 8$ starting tile, depicted in Figure \ref{fig:(2,1)_slant_example}. This starting tile has the first tower at $(0,0)$ and the last tower at $(7,0)$, and there are $5$ total towers required to dominate this grid graph on the lattice. Therefore, we can form a dominating set by moving these towers onto the tile and the $(2,1)$ domination number of the starting tile is at most $5$. We want to use this tile to dominate an $m \times n$ grid and produce an upper bound for $\gamma_{t,r}(S_{m,n})$.  First, when tiling horizontally, we have at most $5$ towers for every $8$ columns of vertices. This implies that if $n= 8q$, the $(2,1)$ domination number of $S_{2,n}$ is bounded by $5q$.

 Since there is a tower in both the bottom left and right corners of a $2\times 8$ tile, adding an additional $2\times 8$ tile horizontally overcounts the number of towers by one. Thus, the domination number of $S_{2,n}$ can be bounded by $5q - (q-1) = 4q + 1$. We then look at tiling vertically. The tile dominates $2$ rows of vertices, so we only need to place a tile for every $2$ rows; therefore, if $m=2p$, we conclude that the $(2,1)$ domination number of $S_{m,n}$ is bounded by $(4q+1)p$, as desired.

If $m = 2p + \ell$ for $\ell = 1$ and $n = 8q + k$ for $k \in \{1, 2, ...,7\}$, we need additional tiles to dominate the remaining columns. We do this by using $q+1$ and $p+1$ in the formula above, giving the second bound. \end{proof}

We give the following bounds for $t$, $r$, $m$, and $n$ using the same methods outlined in Theorem \ref{2,1_slant_tiling} in Table \ref{tbl:upper_slant_bounds}.

\begin{table}[htb!]
\label{upper_bound_table}
\centering
\scalebox{.9}{
\begin{tabular}{|c|c|c|p{1.2in}|p{1.2in}|}
    \hline
     $(t,r)$ & $m$ & $n$ & Upper bound for $\gamma_{t,r}(S_{m,n})$ when $k, \ell = 0$ & Upper bound for $\gamma_{t,r}(S_{m,n})$ when $k, \ell > 0$  \\
     \hline\hline
     $(2,1)$ & $2p + \ell$, $\ell \in \{0,1\}$ & $8q + k, k \in \{0,...,7\}$ & $(4q+1)p$& $(4q+5)(p+1)$ \\
     \hline
     $(3,1)$ & $3p + \ell$, $\ell \in \{0,1,2\}$& $20q + k, k \in \{0,...,19\}$ & $(7q+1)p$& $(7q + 8)(p+1)$ \\
     \hline
     $(3,2)$ & $2p + \ell$, $\ell \in \{0,1\}$& $14q + k, k \in \{0,...,13\}$ &$(6q+1)p$& $(6q+7)(p+1)$\\
     \hline
     $(4,2)$ & $3p + \ell$, $\ell \in \{0,1,2\}$& $15q + k, k \in \{0,...,14\}$ &$(14q+1)p$& $(14q+15)(p+1)$ \\
     \hline
     $(4,3)$ & $2p + \ell, \ell \in \{0,1\}$ & $22q + k, k \in \{0,...,21\}$ &$(8q+1)p$& $(8q+9)(p+1)$ \\
     \hline
     $(5,4)$ & $2p + \ell, \ell \in \{0,1\}$& $32q + k, k \in \{0,...,31\}$ &$(10q+1)p$ & $(10q+11)(p+1)$\\
    \hline
\end{tabular}}
\caption{Upper bounds for $\gamma_{t,r}(S_{m,n})$.}
\label{tbl:upper_slant_bounds}
\end{table}

Further work can be done to tighten the upper bounds given above, and generate more bounds for more choices of $(t,r)$, $m$, and $n$.
We can also use our starting tile strategy to find the $(t,r)$ broadcast domination number for $S_{2,n}$, instead of an upper bound. That process is detailed below. 

\begin{lemma}
If $n \leq 2(t-r)$ and $t > r$, then  $\gamma_{t,r} (S_{2,n}) = 1$. 
\end{lemma}
\begin{proof}
We place the broadcast tower anywhere on the bottom row, and then count the maximum number of columns that are $(t,r)$ dominated by that single broadcast tower. 

To the right of the broadcast tower, the slant diagonal forces vertices in the bottom and top row that are adjacent to the same vertex to have the same reception. So, both vertices in the last column have reception $r$. Therefore, there are $t-r$ columns to the right of the broadcast tower.

On the left side of the tower, signal spreads as it would in a grid, so the vertex in the first column, bottom row will have reception $r+1$, meaning there are $t - r - 1$ columns to the left of the broadcasting tower. 

The total number of columns that are $(t,r)$ dominated by this single tower vertex, including the column the tower is located in, is given by $n = t - r + t  - r - 1 + 1 = 2(t-r)$, as desired.
\end{proof}

\begin{corollary}
For $S_{2,n}$, if $n \leq 2t-2$, then $\gamma_{t,1}(S_{2,n}) = 1$. 
\end{corollary}

\begin{theorem} \label{slant_2xn_start_block}
If $n = 4t-3r$ and $t > r$, placing towers on the top row, $(t-r+1)^{th}$ column and the bottom row, $(n-(t-r))^{th}$ column, creates an efficient dominating set, and $\gamma_{t,r} (S_{2, n}) = 2$.
\end{theorem}
\begin{proof}
\begin{figure}[htb!]
    \centering
    
    \[\begin{tikzpicture}[scale=1.5]
    \foreach \i in {0,1,2,3}{
    \foreach \j in {0}{
    \draw[thick,gray] (\i,\j)--(\i+1,\j+1);
    \draw[thick,gray] (\i,\j)--(\i,\j+1);
    \draw[thick,gray] (\i,\j)--(\i+1,\j);
    \draw[thick, gray] (0,1) -- (1,1) -- (2,1) -- (3,1) -- (4,1) -- (4,0);
    \fill[black] (\i,\j) circle (3pt);
    \fill[black] (\i,\j+1) circle (3pt); 
    \fill[black] (\i+1,\j) circle (3pt);
    \fill[black] (\i+1,\j+1) circle (3pt); 
    }
    }
    \fill[black] (0,1) circle (3pt);
    \fill[black] (2,1) circle (3pt);

    \fill[red] (1,1) circle (3pt);
    \fill[red] (3,0) circle (3pt);
    \node[black] at (.9,1.27) {\Large{$v$}};
    \node[red] at (1.2,.8) {$2$};
    \node[black] at (2.9, -.27) {\Large{$w$}};
    \node[red] at (2.85, .3) {$2$};
    
    \end{tikzpicture} \]

    \caption{Slant grid graph dominated by 2 towers, for $(2,1)$.}
    \label{fig:slant_2xn_startblock_base}
\end{figure}
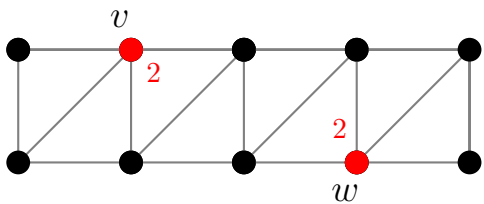

We construct the dominating set on $S_{2,n}$ as shown in the example in Figure~\ref{fig:slant_2xn_startblock_base}, where there is a broadcasting tower, $v$, in the top row, $(t-r+1)^{th}$ column, and a second broadcast tower, $w$, in the bottom row, $(n - (t-r))^{th}$ column. 

We then show that when $\{v,w\}$ forms an efficient dominating set for $S_{m,n}$, it must be that $n = 4t-3r$. To establish this we use a proof by induction.

Let $(2,1)$ be our base case. As shown in Figure \ref{fig:slant_2xn_startblock_base}, we see that the $2 \times 5$ slant grid graph is efficiently dominated by two broadcasting towers, and $5 = 4(2) - 3(1)$. 
So, we assume that $n = 4k-3r$ for all $(k,1)$ with $2\leq k\leq t$. 
We must now show that $n = 4(t+1) - 3$ for $(t+1,1)$. 

By increasing the strength of the broadcasting tower to $t+1$, $n$ must be larger so we can continue to have an efficient dominating set. As a result of increasing the strength to $t+1$, vertices that were distance $t-r$ from a tower now have a signal of $3$, where reception $2$ is from the tower $t-r$ away from that vertex, and reception $1$ is from the other tower, which is $t-r+1$ away. For example, if the towers in Figure \ref{fig:slant_2xn_startblock_base} were strength $3$, the vertices in the third column would have reception $3$, $2$ from the tower they are adjacent to and $1$ from the other tower. So, we can shift the broadcasting towers away from each other. 
We move the one in the top row a column to the left, and shift the second tower a column to the right. This fixes the overlap inside of the grid graph, but the grid graph could have more columns. Let's look at the first broadcasting tower. The vertex we shifted the tower onto used to have reception $t-1$ but now has reception $t+1$, since it is the new tower location. That means its reception has increased by $2$ as we moved between the $(t,1)$ and the $(t+1, 1)$ case. So, there can be $2$ more columns to the left of this broadcasting tower in the grid graph, because vertices in the top and bottom row will have the same reception with the slant. We see a similar result with the second broadcasting tower, so there can also be two more columns added to the right of the second broadcasting tower. So, we have $4$ more columns than we did in the $(t,1)$ case; therefore $n = 4t - 3 + 4 = 4t + 1 = 4(t+1)-3$, as desired.

Now that we have shown that $n = 4t -3$ for all $(t,1)$, we now show $n = 4t -3r$ for all $(t,r)$. Assume that $n = 4t - 3r$ for some fixed $t$. We want to show for $(t, r+1)$ that $n = 4t - 3(r+1)$. Since more reception is required, the towers must be closer together than they were before. 
We keep the position of the first tower fixed and move the second tower one column to the left, so vertices in between the towers have at least reception $r+1$. We then consider the new size of the efficiently dominated grid. 
The vertices in the first column have reception $r$, so we do not want this column in our grid anymore. That means we have one less column to the left of the first broadcast tower. 
Then, we look at the second broadcast tower. 
We already have one less column because we moved the second tower to the left, but we still have a column to the right of the broadcast tower where both vertices have reception $r$. By removing that column, we now have a grid graph with $3$ less columns than before, so for $(t,r+1)$, $n = 4t - 3r - 3 = 4t - 3(r+1)$, as desired. 

Therefore, if $n = 4t - 3r$ and $t > r$, there exists an efficient broadcast dominating set on $S_{2,n}$ with $2$ towers.
\end{proof}

\slantdom*

\begin{proof}
We first construct a broadcast dominating set with cardinality $\left \lceil \frac{2(n+r-1)}{4t-2r-1} \right \rceil$. We use the starting block construction in Theorem \ref{slant_2xn_start_block}. 

\begin{figure}[htb!]
    \centering
    \begin{tikzpicture} 
    
    \draw[thick, black] (0,0)--(0,2);
    \draw[thick, black] (7,0)--(5,0)--(5,2)--(7,2)--cycle;
    \draw[thick, black] (5,0)--(7,2);
    \draw[thick, black] (13,0)--(15,0)--(15,2)--(13,2)--cycle;
    \draw[thick, black] (13,0)--(15,2);
    \draw[dashed, black] (0,0)--(5,0);
    \draw[dashed, black] (0,2)--(5,2);
    \draw[dashed, black] (7,0)--(13,0);
    \draw[dashed, black] (7,2)--(13,2);
    
    \draw[dashed, black] (5,2)--(3,0);
    \draw[dashed, black] (0,0)--(2,2);
    \draw[dashed, black] (7,0)--(9,2);
    \draw[dashed, black] (13,2)--(11,0);
    \draw[dashed, black] (11,2)--(9,0);
    
    \fill[black] (0,0) circle (3pt); 
    \node[black] at (0,-.35) {$r$};
    
    \fill[black] (5,0) circle (3pt); 
    \node[black] at (5,-.35) {$t-1$};
    
    \fill[black] (7,0) circle (3pt); 
    \node[black] at (7,-.35) {$t-1$};
    
    \fill[black] (13,0) circle (3pt); 
    \node[black] at (13,-.35) {$1$};
    
    \fill[black] (15,0) circle (3pt); 
    
    \fill[black] (0,2) circle (3pt); 
    \node[black] at (0,2.35) {$r$};
    
    \fill[black] (5,2) circle (3pt); 
    \node[black] at (5,2.35) {$t-1$};
    
    \fill[red] (7,2) circle (3pt); 
    \node[red] at (7,2.35) {$t$};
    
    \fill[black] (13,2) circle (3pt); 
    \node[black] at (13,2.35) {$2$};
    
    \fill[black] (15,2) circle (3pt); 
    \node[black] at (15,2.35) {$1$};
    
    \end{tikzpicture}
    \caption{Beginning of the constructed dominating set.}
    \label{fig:slant_dom_construction}
\end{figure}
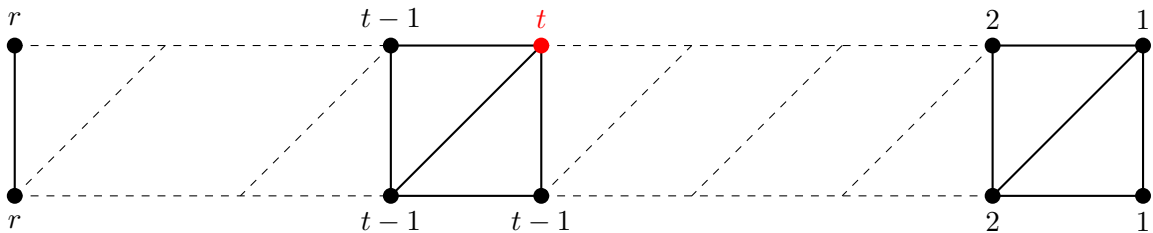

We place our first tower on the top row and $(t-r+1)^{th}$ column, as shown in Figure \ref{fig:slant_dom_construction}. All the vertices to the left of the tower have at least reception $r$, so we only need to continue placing towers to the right of the broadcast tower. To waste as little signal as possible, we add a broadcasting tower in the bottom row, in the  and construct our starting block.
Therefore, we have $4t-3r$ columns of $S_{2,n}$ broadcast dominated. To continue dominating the grid, we want to continue using starting blocks, while allowing starting blocks to overlap so that the vertices between them receive sufficient signal.
The number of such columns between any two broadcast zones is given by $r-1$, which we can conclude from Lemma \ref{vertex_overlap}. To account for this, we partition the grid into blocks with $4t - 3r + r - 1 = 4t - 2r - 1$ columns, except for the rightmost block; there are no towers to the right that contribute signal. Since each block has $2$ broadcasting towers, the size of the dominating set we've constructed is $\left \lceil \frac{2(n+r-1)}{4t-2r-1} \right \rceil$. Thus, $\gamma_{t,r} (S_{2,n}) \leq \left \lceil \frac{2(n+r-1)}{4t-2r-1} \right \rceil$. 

We now show there does not exist a dominating set with less than $\left \lceil \frac{2(n+r-1)}{4t-2r-1} \right \rceil$ towers. We cannot move the first tower any further away from the first column, because then vertices in the first column would have insufficient reception. We also cannot move the towers further away from each other, since we showed in Theorem \ref{slant_2xn_start_block} that the starting block is efficient; if the towers were further apart, vertices between the towers that have reception exactly $r$ would then have insufficient reception. Therefore, there is no way to use fewer towers than are used in our dominating set, so we can conclude \[\gamma_{t,r} (S_{2,n}) = \left \lceil \frac{2(n+r-1)}{4t-2r-1} \right \rceil. \qedhere\]
\end{proof}

\subsection{King's grid graph and lattice}

A king's lattice is a graph that models the possible moves a king can make on an infinite chessboard. A king's grid graph $K_{m,n}$ can be defined as an $m \times n$ grid graph where each vertex is connected to all of its diagonally adjacent vertices. Just like grid graphs, $m$ and $n$ will denote the number of rows and columns of the finite subgraph of the king's lattice, respectively.

The following lemma gives us a general starting block for a king's grid graph of arbitrary finite dimensions.
\begin{lemma} \label{kings_lattice_start_block}
For $t>r$ and $ m \leq 1 + 2(t-r)$, if $n = 4t-3r+1$, then $\gamma_{t,r} (K_{m,n}) = 2$.
\end{lemma}

\begin{remark}
 The first broadcasting tower is at ($t-r+1$, $\left \lceil \frac{m}{2} \right \rceil$), and the second broadcasting tower is at ($n - (t-r)$, $\left \lceil \frac{m}{2} \right \rceil$).
\end{remark}

\begin{proof}
Assume that a starting block has $n=4t-3r+1$ columns. We will induct on $t$, holding $r$ and $m$ fixed, using $(t,r)=(2,1)$ as our base case. Figure \ref{fig:2_1_lattice_example} shows the efficient dominating pattern for $(t,r)=(2,1)$ on $K_{3,n}$. The equality holds, as we have a starting block of $n=4(2)-3(1)+1=6$ columns. Assume $n=4t-3(1)+1 = 4t-2$ holds. Now we will show it holds for $n=4(t+1)-2=4t+2$. If we increase $t$ to $t+1$, overlapped vertices that previously had reception $r$ now have reception $r+3$. To maintain efficiency, we move the broadcasting towers further to the right. We shift the first broadcast vertex one column to the right, so that vertices in the first column have reception $r$; however, vertices between the towers now have $r+4$ reception instead of $r$. So, we must shift the second tower even further to the right. Shifting $4$ columns to the right returns vertices with $r+4$ reception to reception $r$, These observations show that the number of columns in our starting block will expand by $4$ columns.

We will now induct on $r$. Let $(t,r)=(t,1)$ be our base case. Assume $n=4t-3r+1$ holds for $r$, Now we must show $n=4t-3(r+1)+1=4t-3r-2$ holds. Increasing $r$ to $r+1$ will require us to move our first tower one column to the left, so that vertices in the first columns have reception $r$. Vertices between the broadcasting towers that had reception $r$ now have reception $r-1$, so the second broadcasting tower must be shifted $2$ columns to the left. This results in a starting block being reduced by a total of $3$ columns, giving us $n=4t-3r-2$.
\end{proof}

\begin{figure}[htb!]
    \centering
     \begin{tikzpicture}[scale=.5]
\draw[thick, black] (0,0)--(10,0);
\draw[thick, black] (0,2)--(10,2);
\draw[thick, black] (0,4)--(10,4);

\draw[thick, black] (0,0)--(0,4);
\draw[thick, black] (2,4)--(2,0);
\draw[thick, black] (4,4)--(4,0);
\draw[thick, black] (6,4)--(6,0);
\draw[thick, black] (8,4)--(8,0);
\draw[thick, black] (10,4)--(10,0);

\draw[thick, black] (0,0)--(2,2)--(4,4)--(6,2)--(8,0)--(10,2);
\draw[thick, black] (0,4)--(2,2)--(4,0)--(6,2)--(8,4)--(10,2);
\draw[thick, black] (0,2)--(2,4)--(4,2)--(6,0)--(8,2)--(10,4);
\draw[thick, black] (0,2)--(2,0)--(4,2)--(6,4)--(8,2)--(10,0);
\fill[black] (0,0) circle (3pt);
\fill[black] (2,0) circle (3pt);
\fill[black] (4,0) circle (3pt);
\fill[black] (6,0) circle (3pt);
\fill[black] (8,0) circle (3pt);
\fill[black] (10,0) circle (3pt);

\fill[black] (0,2) circle (3pt);
\fill[red] (2,2) circle (7pt); 
\fill[black] (4,2) circle (3pt);
\fill[black] (6,2) circle (3pt);
\fill[blue] (8,2) circle (7pt); 

\fill[black] (10,2) circle (3pt);

\fill[black] (0,4) circle (3pt);
\fill[black] (2,4) circle (3pt);
\fill[black] (4,4) circle (3pt);
\fill[black] (6,4) circle (3pt);
\fill[black] (8,4) circle (3pt);
\fill[black] (10,4) circle (3pt);

\end{tikzpicture}
    \caption{Dominating pattern for $(2,1)$ on $K_{3,n}$.}
    \label{fig:2_1_lattice_example}
\end{figure}
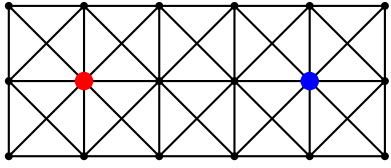

Using Lemma \ref{kings_lattice_start_block}, we find an explicit formula for the domination number on the general finite king's grid graph.
\kingmxn*
\begin{proof}
We first construct the dominating set on the king's grid graph. All towers are placed on row $\left \lceil \frac{m}{2} \right \rceil$, which ensures that the vertices in the rows above and below will have sufficient reception. We place the first broadcasting tower in the $t-r+1$ column, so that all vertices in the first column have reception $r$. We then place the second broadcasting tower so that we've constructed a starting block, as detailed in Lemma \ref{kings_lattice_start_block}. We note that in this construction, the second broadcasting tower is placed on the $2t-r^{th}$ column from the first broadcasting tower. We want to break the rest of the grid graph into these starting blocks, but we must consider the $r-1$ columns of overlap between the second broadcasting tower and the first broadcasting tower of the next block. As in Lemma \ref{vertex_overlap}, we must add $r-1$ to the number of columns $n$ to account for the overlap between the starting blocks. Therefore,
\[\gamma_{t,r}(K_{m,n}) \leq \left \lceil \frac{n+ r-1}{2t-r} \right \rceil.\] 

Now we show that our dominating set or minimum cardinality. We can't move the first tower any farther into the grid graph, since vertices in the first column have reception $r$. We also cannot move towers any further apart, since the starting block is efficient. We cannot move starting blocks farther apart, because the shortest path between the second tower of a block and the first tower of the next block, where each vertex is on a new column, requires $r-1$ vertices of overlap, which we showed in Lemma \ref{vertex_overlap}. Therefore, 
\[\gamma_{t,r}(K_{m,n}) = \left \lceil \frac{n+r-1}{2t-r} \right \rceil. \qedhere \]
\end{proof}

We now provide results on the infinite king's lattice. Here we discuss the spread of a single tower, which we use to find the efficient broadcasting domination patterns for $(t,r) \in \{(t,1),(t,2)\}.$

\begin{remark} \label{kings_broadcast_zone}
The broadcast zone for a tower of strength $t$ in the king's lattice has dimensions $2t-1 \times 2t-1$. The spread of the broadcast zone can be seen in Figure \ref{fig:kings_lattice_spread}.

\end{remark}

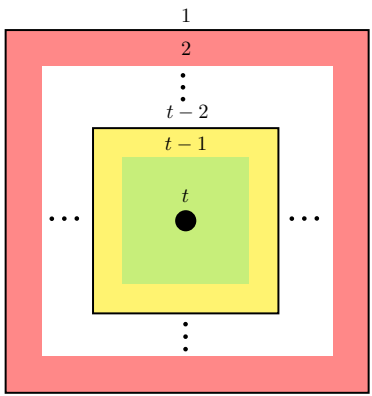
\begin{figure}[htb!]
    \centering
    
\tikzset{every picture/.style={line width=0.75pt}} 

\begin{tikzpicture}[x=0.75pt,y=0.75pt,yscale=-1,xscale=1]

\draw  [fill={rgb, 255:red, 255; green, 0; blue, 0 }  ,fill opacity=0.47 ] (196,48.58) -- (379.08,48.58) -- (379.08,231.67) -- (196,231.67) -- cycle ;
\draw  [draw opacity=0][fill={rgb, 255:red, 255; green, 255; blue, 255 }  ,fill opacity=1 ] (214.38,66.96) -- (360.71,66.96) -- (360.71,213.29) -- (214.38,213.29) -- cycle ;
\draw  [fill={rgb, 255:red, 255; green, 235; blue, 7 }  ,fill opacity=0.58 ] (240.08,98.08) -- (333.58,98.08) -- (333.58,191.58) -- (240.08,191.58) -- cycle ;
\draw  [draw opacity=0][fill={rgb, 255:red, 137; green, 233; blue, 134 }  ,fill opacity=0.47 ] (254.96,112.96) -- (318.71,112.96) -- (318.71,176.71) -- (254.96,176.71) -- cycle ;
\draw  [fill={rgb, 255:red, 0; green, 0; blue, 0 }  ,fill opacity=1 ] (282,144.83) .. controls (282,142.16) and (284.16,140) .. (286.83,140) .. controls (289.5,140) and (291.67,142.16) .. (291.67,144.83) .. controls (291.67,147.5) and (289.5,149.67) .. (286.83,149.67) .. controls (284.16,149.67) and (282,147.5) .. (282,144.83) -- cycle ;
\draw  [fill={rgb, 255:red, 0; green, 0; blue, 0 }  ,fill opacity=1 ] (219.32,144.46) .. controls (218.96,144.47) and (218.66,144.19) .. (218.64,143.83) .. controls (218.63,143.47) and (218.91,143.17) .. (219.27,143.16) .. controls (219.63,143.14) and (219.93,143.42) .. (219.95,143.78) .. controls (219.96,144.14) and (219.68,144.44) .. (219.32,144.46) -- cycle ;
\draw  [fill={rgb, 255:red, 0; green, 0; blue, 0 }  ,fill opacity=1 ] (225.62,144.4) .. controls (225.26,144.42) and (224.96,144.14) .. (224.95,143.78) .. controls (224.93,143.42) and (225.21,143.12) .. (225.57,143.1) .. controls (225.93,143.09) and (226.23,143.37) .. (226.25,143.73) .. controls (226.26,144.09) and (225.98,144.39) .. (225.62,144.4) -- cycle ;
\draw  [fill={rgb, 255:red, 0; green, 0; blue, 0 }  ,fill opacity=1 ] (231.93,144.35) .. controls (231.57,144.37) and (231.26,144.09) .. (231.25,143.73) .. controls (231.23,143.37) and (231.51,143.06) .. (231.87,143.05) .. controls (232.23,143.03) and (232.54,143.31) .. (232.55,143.67) .. controls (232.57,144.03) and (232.29,144.34) .. (231.93,144.35) -- cycle ;
\draw  [fill={rgb, 255:red, 0; green, 0; blue, 0 }  ,fill opacity=1 ] (340.37,144.41) .. controls (340.01,144.43) and (339.7,144.15) .. (339.69,143.79) .. controls (339.67,143.43) and (339.95,143.13) .. (340.31,143.11) .. controls (340.67,143.1) and (340.98,143.38) .. (340.99,143.74) .. controls (341.01,144.1) and (340.73,144.4) .. (340.37,144.41) -- cycle ;
\draw  [fill={rgb, 255:red, 0; green, 0; blue, 0 }  ,fill opacity=1 ] (346.67,144.36) .. controls (346.31,144.37) and (346.01,144.1) .. (345.99,143.74) .. controls (345.98,143.38) and (346.25,143.07) .. (346.61,143.06) .. controls (346.97,143.04) and (347.28,143.32) .. (347.29,143.68) .. controls (347.31,144.04) and (347.03,144.34) .. (346.67,144.36) -- cycle ;
\draw  [fill={rgb, 255:red, 0; green, 0; blue, 0 }  ,fill opacity=1 ] (352.97,144.31) .. controls (352.61,144.32) and (352.31,144.04) .. (352.29,143.68) .. controls (352.28,143.32) and (352.56,143.02) .. (352.92,143) .. controls (353.28,142.99) and (353.58,143.27) .. (353.6,143.63) .. controls (353.61,143.99) and (353.33,144.29) .. (352.97,144.31) -- cycle ;
\draw  [fill={rgb, 255:red, 0; green, 0; blue, 0 }  ,fill opacity=1 ] (287.4,209.61) .. controls (287.41,209.97) and (287.13,210.27) .. (286.77,210.28) .. controls (286.41,210.29) and (286.11,210.01) .. (286.09,209.65) .. controls (286.08,209.29) and (286.36,208.99) .. (286.72,208.98) .. controls (287.08,208.96) and (287.38,209.25) .. (287.4,209.61) -- cycle ;
\draw  [fill={rgb, 255:red, 0; green, 0; blue, 0 }  ,fill opacity=1 ] (287.39,203.3) .. controls (287.4,203.66) and (287.12,203.96) .. (286.76,203.98) .. controls (286.4,203.99) and (286.1,203.71) .. (286.09,203.35) .. controls (286.08,202.99) and (286.36,202.68) .. (286.72,202.67) .. controls (287.08,202.66) and (287.38,202.94) .. (287.39,203.3) -- cycle ;
\draw  [fill={rgb, 255:red, 0; green, 0; blue, 0 }  ,fill opacity=1 ] (287.39,197) .. controls (287.4,197.36) and (287.12,197.66) .. (286.76,197.67) .. controls (286.4,197.68) and (286.1,197.4) .. (286.08,197.04) .. controls (286.07,196.68) and (286.35,196.38) .. (286.71,196.37) .. controls (287.07,196.36) and (287.37,196.64) .. (287.39,197) -- cycle ;
\draw  [fill={rgb, 255:red, 0; green, 0; blue, 0 }  ,fill opacity=1 ] (284.8,71.49) .. controls (284.78,71.13) and (285.05,70.82) .. (285.41,70.8) .. controls (285.77,70.78) and (286.08,71.06) .. (286.1,71.42) .. controls (286.12,71.78) and (285.85,72.08) .. (285.49,72.11) .. controls (285.13,72.13) and (284.82,71.85) .. (284.8,71.49) -- cycle ;
\draw  [fill={rgb, 255:red, 0; green, 0; blue, 0 }  ,fill opacity=1 ] (284.8,77.49) .. controls (284.78,77.13) and (285.05,76.82) .. (285.41,76.8) .. controls (285.77,76.78) and (286.08,77.06) .. (286.1,77.42) .. controls (286.12,77.78) and (285.85,78.08) .. (285.49,78.11) .. controls (285.13,78.13) and (284.82,77.85) .. (284.8,77.49) -- cycle ;
\draw  [fill={rgb, 255:red, 0; green, 0; blue, 0 }  ,fill opacity=1 ] (284.8,83.49) .. controls (284.78,83.13) and (285.05,82.82) .. (285.41,82.8) .. controls (285.77,82.78) and (286.08,83.06) .. (286.1,83.42) .. controls (286.12,83.78) and (285.85,84.08) .. (285.49,84.11) .. controls (285.13,84.13) and (284.82,83.85) .. (284.8,83.49) -- cycle ;

\draw (286.67,132) node [scale=0.7]  {$t$};
\draw (287,106) node [scale=0.7]  {$t-1$};
\draw (287.71,89.86) node [scale=0.7]  {$t-2$};
\draw (287,58) node [scale=0.7]  {$2$};
\draw (287,41) node [scale=0.7]  {$1$};

\end{tikzpicture}

    
    
    \caption{Spread of a tower in the king's lattice.}
    \label{fig:kings_lattice_spread}
\end{figure}
 Now we define a method for determining the shortest path between two vertices in the king's lattice. The following lemma is used to develop our results for the efficient broadcasting domination patterns for $(t,1)$ and $(t,2)$. We let $d_K((x,y), (p,q))$ denote the distance between two points in the king's lattice.
\begin{lemma}\label{fig:distance}
The distance between two points $(x,y)$ and a point $(p,q)$ in the king's lattice is $d_K((x,y),(p,q)) = \max(|p-x|,|q-y|)$.
\end{lemma}

\begin{proof}
First consider Figure \ref{fig:king1x1} when dealing with a $1\times 1$ section of the king's lattice. The shortest path would be across the red dotted path, which is of length $1$. The other distance between these vertices would be $2$, which is achieved by going along the blue dashed path.
\begin{figure}[htb!]
    \centering
\resizebox{1in}{!}{

\tikzset{every picture/.style={line width=0.75pt}} 

\begin{tikzpicture}[x=0.75pt,y=0.75pt,yscale=-1,xscale=1]

\draw [line width=1.5]    (213.67,66.83) -- (360,213.17) ;

\draw  [color={rgb, 255:red, 0; green, 0; blue, 0 }  ,draw opacity=1 ][line width=1.5]  (213.67,66.83) -- (360,66.83) -- (360,213.17) -- (213.67,213.17) -- cycle ;
\draw [color={rgb, 255:red, 74; green, 144; blue, 226 }  ,draw opacity=1, dashed][line width=4]    (213.67,213.17) -- (360,213.17) -- (360,66.83) ;

\draw [color={rgb, 255:red, 0; green, 0; blue, 0 }  ,draw opacity=1][line width=1.5]    (360,66.83) -- (213.67,213.17);

\draw [color={rgb, 255:red, 243; green, 21; blue, 21 }  ,draw opacity=1 , dotted][line width=4]    (360,66.83) -- (213.67,213.17) ;

\draw  [fill={rgb, 255:red, 0; green, 0; blue, 0 }  ,fill opacity=1 ] (207.42,213.17) .. controls (207.42,209.71) and (210.21,206.92) .. (213.67,206.92) .. controls (217.12,206.92) and (219.92,209.71) .. (219.92,213.17) .. controls (219.92,216.62) and (217.12,219.42) .. (213.67,219.42) .. controls (210.21,219.42) and (207.42,216.62) .. (207.42,213.17) -- cycle ;
\draw  [fill={rgb, 255:red, 0; green, 0; blue, 0 }  ,fill opacity=1 ] (353.75,213.17) .. controls (353.75,209.71) and (356.55,206.92) .. (360,206.92) .. controls (363.45,206.92) and (366.25,209.71) .. (366.25,213.17) .. controls (366.25,216.62) and (363.45,219.42) .. (360,219.42) .. controls (356.55,219.42) and (353.75,216.62) .. (353.75,213.17) -- cycle ;
\draw  [fill={rgb, 255:red, 0; green, 0; blue, 0 }  ,fill opacity=1 ] (207.42,66.83) .. controls (207.42,63.38) and (210.21,60.58) .. (213.67,60.58) .. controls (217.12,60.58) and (219.92,63.38) .. (219.92,66.83) .. controls (219.92,70.29) and (217.12,73.08) .. (213.67,73.08) .. controls (210.21,73.08) and (207.42,70.29) .. (207.42,66.83) -- cycle ;
\draw  [fill={rgb, 255:red, 0; green, 0; blue, 0 }  ,fill opacity=1 ] (353.75,66.83) .. controls (353.75,63.38) and (356.55,60.58) .. (360,60.58) .. controls (363.45,60.58) and (366.25,63.38) .. (366.25,66.83) .. controls (366.25,70.29) and (363.45,73.08) .. (360,73.08) .. controls (356.55,73.08) and (353.75,70.29) .. (353.75,66.83) -- cycle ;

\draw (289,120) node   {\Large{$1$}};
\draw (289,228) node   {\Large{$1$}};
\draw (372,130) node   {\Large{$1$}};

\end{tikzpicture}}
    \caption{$K_{1\times1}$}
    \label{fig:king1x1}
\end{figure}
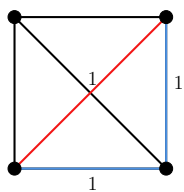

Without loss of generality, assume $p \geq x, q \geq y$, and $|p-x| \leq |q-y|$. Recall that $d_K((x,y),(p,q))$ is the length of the shortest path  between the points $(x,y)$ and $(p,q)$ in the king's lattice. A path can be constructed between two points by using as many diagonals as possible, and then moving strictly vertically the remainder of the way. The maximum number of diagonals that can be used in the path from $(x,y)$ to $(p,q)$, is $|p-x|$; otherwise the path passes the column $q$. The remaining vertical distance to the point $(p,q)$ is then $|q-y| - |p-x|$, and the final path is length $|q-y|$. If a shorter path existed, it could not be formed of only diagonal and vertical steps, or else it would again be length $|q-y|$. However, it would then need to use horizontal steps, which would then need to be accompanied by vertical steps to reach row $q$. However, using a diagonal step is shorter than a horizontal and then vertical step, so a shorter path cannot exist. 
\end{proof}

\begin{figure}[htb!]
    \centering
\tikzset{every picture/.style={line width=0.75pt}} 

\begin{tikzpicture}[x=0.75pt,y=0.75pt,yscale=-1,xscale=1]

\draw  [fill={rgb, 255:red, 0; green, 0; blue, 0 }  ,fill opacity=1 ] (134.5,67.38) .. controls (134.5,65.51) and (136.01,64) .. (137.88,64) .. controls (139.74,64) and (141.25,65.51) .. (141.25,67.38) .. controls (141.25,69.24) and (139.74,70.75) .. (137.88,70.75) .. controls (136.01,70.75) and (134.5,69.24) .. (134.5,67.38) -- cycle ;
\draw  [fill={rgb, 255:red, 0; green, 0; blue, 0 }  ,fill opacity=1 ] (55,137.88) .. controls (55,136.01) and (56.51,134.5) .. (58.38,134.5) .. controls (60.24,134.5) and (61.75,136.01) .. (61.75,137.88) .. controls (61.75,139.74) and (60.24,141.25) .. (58.38,141.25) .. controls (56.51,141.25) and (55,139.74) .. (55,137.88) -- cycle ;
\draw    (137.88,67.38) -- (58.38,137.88) ;

\draw   (56,148) .. controls (56,152.67) and (58.33,155) .. (63,155) -- (89.75,155) .. controls (96.42,155) and (99.75,157.33) .. (99.75,162) .. controls (99.75,157.33) and (103.08,155) .. (109.75,155)(106.75,155) -- (136.5,155) .. controls (141.17,155) and (143.5,152.67) .. (143.5,148) ;
\draw   (156.5,140) .. controls (161.17,140.07) and (163.53,137.77) .. (163.6,133.1) -- (163.88,112.85) .. controls (163.97,106.18) and (166.35,102.88) .. (171.02,102.95) .. controls (166.35,102.88) and (164.07,99.52) .. (164.16,92.86)(164.12,95.86) -- (164.41,75.1) .. controls (164.47,70.43) and (162.17,68.07) .. (157.51,68) ;
\draw  [fill={rgb, 255:red, 0; green, 0; blue, 0 }  ,fill opacity=1 ] (353.5,71.38) .. controls (353.5,69.51) and (355.01,68) .. (356.88,68) .. controls (358.74,68) and (360.25,69.51) .. (360.25,71.38) .. controls (360.25,73.24) and (358.74,74.75) .. (356.88,74.75) .. controls (355.01,74.75) and (353.5,73.24) .. (353.5,71.38) -- cycle ;
\draw  [fill={rgb, 255:red, 0; green, 0; blue, 0 }  ,fill opacity=1 ] (269,138.88) .. controls (269,137.01) and (270.51,135.5) .. (272.38,135.5) .. controls (274.24,135.5) and (275.75,137.01) .. (275.75,138.88) .. controls (275.75,140.74) and (274.24,142.25) .. (272.38,142.25) .. controls (270.51,142.25) and (269,140.74) .. (269,138.88) -- cycle ;
\draw    (356.88,71.38) -- (309.5,71) -- (272.38,138.88) ;

\draw   (270,149) .. controls (270,153.67) and (272.33,156) .. (277,156) -- (303.75,156) .. controls (310.42,156) and (313.75,158.33) .. (313.75,163) .. controls (313.75,158.33) and (317.08,156) .. (323.75,156)(320.75,156) -- (350.5,156) .. controls (355.17,156) and (357.5,153.67) .. (357.5,149) ;
\draw   (370.5,141) .. controls (375.17,141.07) and (377.53,138.77) .. (377.6,134.1) -- (377.88,113.85) .. controls (377.97,107.18) and (380.35,103.88) .. (385.02,103.95) .. controls (380.35,103.88) and (378.07,100.52) .. (378.16,93.86)(378.12,96.86) -- (378.41,76.1) .. controls (378.47,71.43) and (376.17,69.07) .. (371.51,69) ;
\draw  [fill={rgb, 255:red, 0; green, 0; blue, 0 }  ,fill opacity=1 ] (306.13,71.38) .. controls (306.13,69.51) and (307.64,68) .. (309.5,68) .. controls (311.36,68) and (312.88,69.51) .. (312.88,71.38) .. controls (312.88,73.24) and (311.36,74.75) .. (309.5,74.75) .. controls (307.64,74.75) and (306.13,73.24) .. (306.13,71.38) -- cycle ;
\draw  [fill={rgb, 255:red, 0; green, 0; blue, 0 }  ,fill opacity=1 ] (564.13,112.63) .. controls (564.13,110.76) and (565.64,109.25) .. (567.5,109.25) .. controls (569.36,109.25) and (570.88,110.76) .. (570.88,112.63) .. controls (570.88,114.49) and (569.36,116) .. (567.5,116) .. controls (565.64,116) and (564.13,114.49) .. (564.13,112.63) -- cycle ;
\draw  [fill={rgb, 255:red, 0; green, 0; blue, 0 }  ,fill opacity=1 ] (485,138.88) .. controls (485,137.01) and (486.51,135.5) .. (488.38,135.5) .. controls (490.24,135.5) and (491.75,137.01) .. (491.75,138.88) .. controls (491.75,140.74) and (490.24,142.25) .. (488.38,142.25) .. controls (486.51,142.25) and (485,140.74) .. (485,138.88) -- cycle ;
\draw    (566.5,113) -- (487.38,140.88) ;

\draw   (486,149) .. controls (486,153.67) and (488.33,156) .. (493,156) -- (519.75,156) .. controls (526.42,156) and (529.75,158.33) .. (529.75,163) .. controls (529.75,158.33) and (533.08,156) .. (539.75,156)(536.75,156) -- (566.5,156) .. controls (571.17,156) and (573.5,153.67) .. (573.5,149) ;
\draw   (586.5,141) .. controls (591.17,141.07) and (593.53,138.77) .. (593.6,134.1) -- (593.88,113.85) .. controls (593.97,107.18) and (596.35,103.88) .. (601.02,103.95) .. controls (596.35,103.88) and (594.07,100.52) .. (594.16,93.86)(594.12,96.86) -- (594.41,76.1) .. controls (594.47,71.43) and (592.17,69.07) .. (587.51,69) ;
\draw    (569.5,70) -- (568.5,112) ;

\draw  [fill={rgb, 255:red, 0; green, 0; blue, 0 }  ,fill opacity=1 ] (566.13,70) .. controls (566.13,68.14) and (567.64,66.63) .. (569.5,66.63) .. controls (571.36,66.63) and (572.88,68.14) .. (572.88,70) .. controls (572.88,71.86) and (571.36,73.38) .. (569.5,73.38) .. controls (567.64,73.38) and (566.13,71.86) .. (566.13,70) -- cycle ;

\draw (100,172) node   {$| p-x| $};
\draw (205,103) node   {$| q-y| $};
\draw (140,48) node   {$( x,y)$};
\draw (31,138) node   {$( p,q)$};
\draw (314,173) node   {$| p-x| $};
\draw (413,104) node   {$| q-y| $};
\draw (354,49) node   {$( x,y)$};
\draw (248,139) node   {$( p,q)$};
\draw (530,173) node   {$| p-x| $};
\draw (635,104) node   {$| q-y| $};
\draw (570,49) node   {$( x,y)$};
\draw (461,139) node   {$( p,q)$};
\draw (98,205) node   {$( 1)$};
\draw (314,205) node   {$( 2)$};
\draw (530,206) node   {$( 3)$};

\end{tikzpicture}
    \caption{Different cases for distance from two points.}
    \label{fig:shortestpathk}
\end{figure}
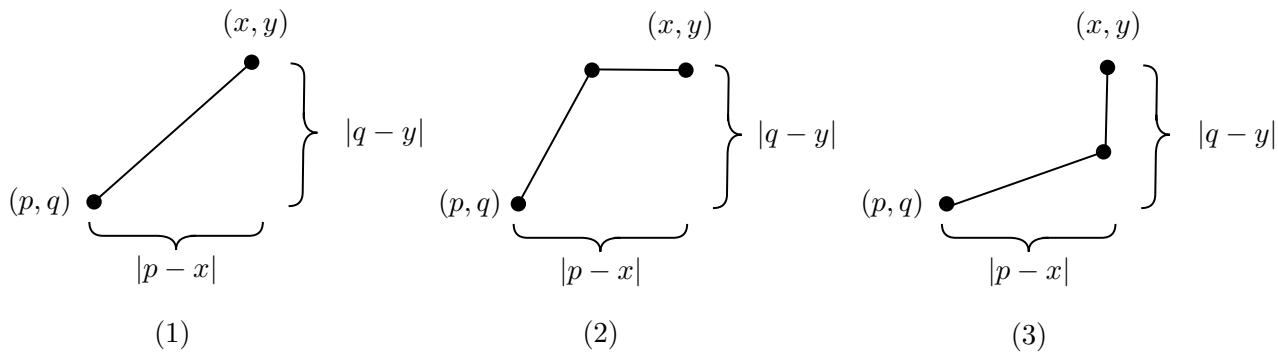

We now find an efficient dominating pattern on the king's lattice for varying $t$ and $r=1$.
\begin{theorem}
Let $t>1$ and $r=1$. Then an efficient $(t,1)$ broadcast domination pattern for the infinite king's lattice is given by placing towers at every vertex of the form
\[
v_{x,y}=(x(2t-1),y(2t-1)),
\]
where $x,y \in \ZZ$.
\end{theorem}
\begin{proof}
We must show that placing towers at each of the vertices of $(x(2t-1),y(2t-1))$ dominates the infinite king's lattice efficiently for $(t,1)$. We do so by using Remark \ref{kings_broadcast_zone} and by calculating distances from vertices to towers.  In Figure
\ref{fig:(t,1)for K_infty}, each green square represents the broadcast zone of each of the towers, which are represented by the black nodes. The green shading represents all the vertices within the broadcast zone that receive sufficient signal. Since we tile the 
lattice using these towers' broadcast zones and the tiling is uniform throughout the lattice, we can focus our attention to the following $4$ towers placed at positions
\begin{align*}
& v_0 \text{ at } (0,0) \text{ where } x=0,y=0, \\
& v_1 \text{ at } (2t-1,0) \text{ where } x=1,y=0,\\
& v_3 \text{ at } (0, 2t-1) \text{ where } x=0,y=1,\\
& v_2 \text{ at } (2t-1,2t-1) \text{ where } x=1,y=1.\\
\end{align*}
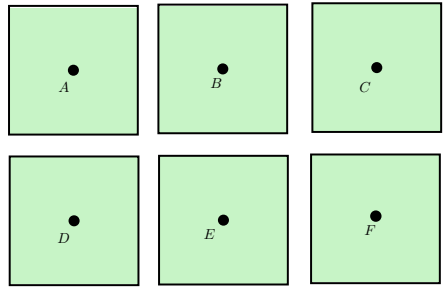
\begin{figure}[htb!]
    \centering
\tikzset{every picture/.style={line width=0.75pt}} 
\[
\begin{tikzpicture}[x=0.75pt,y=0.75pt,yscale=-1,xscale=1]

\draw  [draw opacity=0][fill={rgb, 255:red, 137; green, 233; blue, 134 }  ,fill opacity=0.47 ] (134.94,168) -- (198.69,168) -- (198.69,231.75) -- (134.94,231.75) -- cycle ;
\draw  [color={rgb, 255:red, 0; green, 0; blue, 0 }  ,draw opacity=1 ][fill={rgb, 255:red, 0; green, 0; blue, 0 }  ,fill opacity=1 ] (164.57,199.36) .. controls (164.57,198.09) and (165.59,197.07) .. (166.86,197.07) .. controls (168.12,197.07) and (169.14,198.09) .. (169.14,199.36) .. controls (169.14,200.62) and (168.12,201.64) .. (166.86,201.64) .. controls (165.59,201.64) and (164.57,200.62) .. (164.57,199.36) -- cycle ;
\draw  [color={rgb, 255:red, 0; green, 0; blue, 0 }  ,draw opacity=1 ] (134.38,166.88) -- (199.33,166.88) -- (199.33,231.83) -- (134.38,231.83) -- cycle ;
\draw  [draw opacity=0][fill={rgb, 255:red, 137; green, 233; blue, 134 }  ,fill opacity=0.47 ] (210.27,167) -- (274.52,167) -- (274.52,231.25) -- (210.27,231.25) -- cycle ;
\draw  [color={rgb, 255:red, 0; green, 0; blue, 0 }  ,draw opacity=1 ][fill={rgb, 255:red, 0; green, 0; blue, 0 }  ,fill opacity=1 ] (239.9,198.69) .. controls (239.9,197.43) and (240.93,196.4) .. (242.19,196.4) .. controls (243.45,196.4) and (244.48,197.43) .. (244.48,198.69) .. controls (244.48,199.95) and (243.45,200.98) .. (242.19,200.98) .. controls (240.93,200.98) and (239.9,199.95) .. (239.9,198.69) -- cycle ;
\draw  [color={rgb, 255:red, 0; green, 0; blue, 0 }  ,draw opacity=1 ] (209.72,166.22) -- (274.66,166.22) -- (274.66,231.16) -- (209.72,231.16) -- cycle ;
\draw  [draw opacity=0][fill={rgb, 255:red, 137; green, 233; blue, 134 }  ,fill opacity=0.47 ] (287.88,166.11) -- (352.25,166.11) -- (352.25,230.48) -- (287.88,230.48) -- cycle ;
\draw  [color={rgb, 255:red, 0; green, 0; blue, 0 }  ,draw opacity=1 ][fill={rgb, 255:red, 0; green, 0; blue, 0 }  ,fill opacity=1 ] (317.57,198.02) .. controls (317.57,196.76) and (318.59,195.74) .. (319.86,195.74) .. controls (321.12,195.74) and (322.14,196.76) .. (322.14,198.02) .. controls (322.14,199.29) and (321.12,200.31) .. (319.86,200.31) .. controls (318.59,200.31) and (317.57,199.29) .. (317.57,198.02) -- cycle ;
\draw  [color={rgb, 255:red, 0; green, 0; blue, 0 }  ,draw opacity=1 ] (287.38,165.55) -- (352.33,165.55) -- (352.33,230.5) -- (287.38,230.5) -- cycle ;
\draw  [draw opacity=0][fill={rgb, 255:red, 137; green, 233; blue, 134 }  ,fill opacity=0.47 ] (135.25,242.98) -- (199.55,242.98) -- (199.55,307.27) -- (135.25,307.27) -- cycle ;
\draw  [color={rgb, 255:red, 0; green, 0; blue, 0 }  ,draw opacity=1 ][fill={rgb, 255:red, 0; green, 0; blue, 0 }  ,fill opacity=1 ] (164.9,275.36) .. controls (164.9,274.09) and (165.93,273.07) .. (167.19,273.07) .. controls (168.45,273.07) and (169.48,274.09) .. (169.48,275.36) .. controls (169.48,276.62) and (168.45,277.64) .. (167.19,277.64) .. controls (165.93,277.64) and (164.9,276.62) .. (164.9,275.36) -- cycle ;
\draw  [color={rgb, 255:red, 0; green, 0; blue, 0 }  ,draw opacity=1 ] (134.72,242.88) -- (199.66,242.88) -- (199.66,307.83) -- (134.72,307.83) -- cycle ;
\draw  [draw opacity=0][fill={rgb, 255:red, 137; green, 233; blue, 134 }  ,fill opacity=0.47 ] (209.69,242.75) -- (274.44,242.75) -- (274.44,307.5) -- (209.69,307.5) -- cycle ;
\draw  [color={rgb, 255:red, 0; green, 0; blue, 0 }  ,draw opacity=1 ][fill={rgb, 255:red, 0; green, 0; blue, 0 }  ,fill opacity=1 ] (240.24,275.02) .. controls (240.24,273.76) and (241.26,272.74) .. (242.52,272.74) .. controls (243.79,272.74) and (244.81,273.76) .. (244.81,275.02) .. controls (244.81,276.29) and (243.79,277.31) .. (242.52,277.31) .. controls (241.26,277.31) and (240.24,276.29) .. (240.24,275.02) -- cycle ;
\draw  [color={rgb, 255:red, 0; green, 0; blue, 0 }  ,draw opacity=1 ] (210.05,242.55) -- (275,242.55) -- (275,307.5) -- (210.05,307.5) -- cycle ;
\draw  [draw opacity=0][fill={rgb, 255:red, 137; green, 233; blue, 134 }  ,fill opacity=0.47 ] (287.55,242.44) -- (351.25,242.44) -- (351.25,306.14) -- (287.55,306.14) -- cycle ;
\draw  [color={rgb, 255:red, 0; green, 0; blue, 0 }  ,draw opacity=1 ][fill={rgb, 255:red, 0; green, 0; blue, 0 }  ,fill opacity=1 ] (317.01,272.99) .. controls (317.01,274.3) and (318.07,275.36) .. (319.38,275.36) .. controls (320.69,275.36) and (321.75,274.3) .. (321.75,272.99) .. controls (321.75,271.68) and (320.69,270.62) .. (319.38,270.62) .. controls (318.07,270.62) and (317.01,271.68) .. (317.01,272.99) -- cycle ;
\draw  [color={rgb, 255:red, 0; green, 0; blue, 0 }  ,draw opacity=1 ] (286.72,241.88) -- (351.66,241.88) -- (351.66,306.83) -- (286.72,306.83) -- cycle ;

\draw (162,208) node [scale=.75]  {$v_3$};
\draw (239,206) node [scale=.75]  {$v_2$};
\draw (162,284) node [scale=.75]  {$v_0$};
\draw (235.5,282) node [scale=.75]  {$v_1$};

\end{tikzpicture}\]
\caption{Plotting $6$ towers in king's lattice for $(t,1)$}
    \label{fig:(t,1)for K_infty}
\end{figure}

Additional towers are placed to illustrate the tiling of the lattice. 
We confirm that this is an efficient broadcast domination pattern. Note that the distance between any two broadcast towers is at least $2t-1$, with $v_0, v_2, v_2$ and $v_3$ all being exactly distance $2t - 1$ from each other. From Remark \ref{kings_broadcast_zone}, the broadcast zone of a tower is $2t-1 \times 2t-1$, so there cannot be any vertices that do not receive signal from at least one broadcast tower. For efficiency, without loss of generality consider $(x,y)$ in the broadcast zone of $v_0$. From Remark \ref{kings_broadcast_zone}, $d_K((x,y), v_0) \leq t - 1$. We then have

\begin{align*}
d_K((x,y), v_1) &\geq |d_K((x,y), v_0)) - d(v_0, v_1)| \geq 2t - 1 - (t-1) = t.\\
\end{align*}
\end{proof}

We can repeat this for $v_2$ and $v_3$ to conclude that if $(x,y)$ is in a broadcast zone, then it is only in that broadcast zone. As no overlap is required for $(t,1)$ domination, our broadcast domination pattern is efficient. 

  
We have the following result for varying $t$ and $r=2$ on the king's Lattice.
\begin{theorem}
Let $t>1$ and $r = 2$. Then an efficient $(t,2)$ broadcasting domination pattern for the infinite king's lattice is given by placing towers at every vertex of the form
\[
((2t-r)x-y, x+(2t-r)y)
\]
with $x,y \in \ZZ$.
\end{theorem}
\begin{proof}
We must show that placing towers at each of the vertices of $((2t-r)x-y, x+(2t-r)y)$ dominates the infinite king's Lattice efficiently for $(t,2)$. We do so by using Remark \ref{kings_broadcast_zone} and by calculating distances from vertices to towers. Observe in Figure \ref{fig:(t,2) king lattice}, each green overlap represents all the vertices that receive at least signal $2$ and the red border of the each broadcast zone represents the vertices that receive signal $1$. Now, for all the vertices in each of the broadcast zones have at least reception $2$, the borders of each broadcast zone must overlap. In Figure \ref{fig:(t,2) king lattice}, we plot $6$ towers to demonstrate the pattern using $((2t-r)x-y, x+(2t-r)y)$:
\begin{align*}
& A \text{ at } (0,0) \text{ where } x=0,y=0, \\
& B \text{ at } ((2t-r),1) \text{ where } x=1,y=0,\\
& C \text{ at } (2(2t-r),2) \text{ where } x=2,y=0,\\
& D \text{ at } (1,-(2t-r)) \text{ where } x=0,y=-1,\\
& E \text{ at } ((2t-r)+1,1-(2t-r)) \text{ where } x=1,y=-1,\\
& F \text{ at } (2(2t-r),2-(2t-r)) \text{ where } x=2,y=-1.
\end{align*}

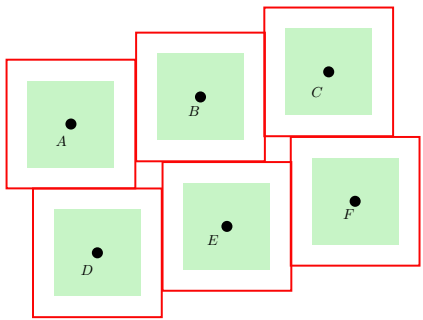
\begin{figure}[htb!]
    \centering
\tikzset{every picture/.style={line width=0.75pt}} 
\[
\begin{tikzpicture}[x=0.75pt,y=0.75pt,yscale=-1,xscale=1]

\draw  [draw opacity=0][fill={rgb, 255:red, 137; green, 233; blue, 134 }  ,fill opacity=0.47 ] (157.94,190.44) -- (201.77,190.44) -- (201.77,234.27) -- (157.94,234.27) -- cycle ;
\draw  [fill={rgb, 255:red, 0; green, 0; blue, 0 }  ,fill opacity=1 ] (177.57,212.36) .. controls (177.57,211.09) and (178.59,210.07) .. (179.86,210.07) .. controls (181.12,210.07) and (182.14,211.09) .. (182.14,212.36) .. controls (182.14,213.62) and (181.12,214.64) .. (179.86,214.64) .. controls (178.59,214.64) and (177.57,213.62) .. (177.57,212.36) -- cycle ;
\draw  [color={rgb, 255:red, 250; green, 6; blue, 6 }  ,draw opacity=1 ] (147.38,179.88) -- (212.33,179.88) -- (212.33,244.83) -- (147.38,244.83) -- cycle ;
\draw  [draw opacity=0][fill={rgb, 255:red, 137; green, 233; blue, 134 }  ,fill opacity=0.47 ] (223.27,176.77) -- (267.11,176.77) -- (267.11,220.61) -- (223.27,220.61) -- cycle ;
\draw  [fill={rgb, 255:red, 0; green, 0; blue, 0 }  ,fill opacity=1 ] (242.9,198.69) .. controls (242.9,197.43) and (243.93,196.4) .. (245.19,196.4) .. controls (246.45,196.4) and (247.48,197.43) .. (247.48,198.69) .. controls (247.48,199.95) and (246.45,200.98) .. (245.19,200.98) .. controls (243.93,200.98) and (242.9,199.95) .. (242.9,198.69) -- cycle ;
\draw  [color={rgb, 255:red, 250; green, 6; blue, 6 }  ,draw opacity=1 ] (212.72,166.22) -- (277.66,166.22) -- (277.66,231.16) -- (212.72,231.16) -- cycle ;
\draw  [draw opacity=0][fill={rgb, 255:red, 137; green, 233; blue, 134 }  ,fill opacity=0.47 ] (287.94,164.11) -- (331.77,164.11) -- (331.77,207.94) -- (287.94,207.94) -- cycle ;
\draw  [fill={rgb, 255:red, 0; green, 0; blue, 0 }  ,fill opacity=1 ] (307.57,186.02) .. controls (307.57,184.76) and (308.59,183.74) .. (309.86,183.74) .. controls (311.12,183.74) and (312.14,184.76) .. (312.14,186.02) .. controls (312.14,187.29) and (311.12,188.31) .. (309.86,188.31) .. controls (308.59,188.31) and (307.57,187.29) .. (307.57,186.02) -- cycle ;
\draw  [color={rgb, 255:red, 250; green, 6; blue, 6 }  ,draw opacity=1 ] (277.38,153.55) -- (342.33,153.55) -- (342.33,218.5) -- (277.38,218.5) -- cycle ;
\draw  [draw opacity=0][fill={rgb, 255:red, 137; green, 233; blue, 134 }  ,fill opacity=0.47 ] (171.27,255.44) -- (215.11,255.44) -- (215.11,299.27) -- (171.27,299.27) -- cycle ;
\draw  [fill={rgb, 255:red, 0; green, 0; blue, 0 }  ,fill opacity=1 ] (190.9,277.36) .. controls (190.9,276.09) and (191.93,275.07) .. (193.19,275.07) .. controls (194.45,275.07) and (195.48,276.09) .. (195.48,277.36) .. controls (195.48,278.62) and (194.45,279.64) .. (193.19,279.64) .. controls (191.93,279.64) and (190.9,278.62) .. (190.9,277.36) -- cycle ;
\draw  [color={rgb, 255:red, 250; green, 6; blue, 6 }  ,draw opacity=1 ] (160.72,244.88) -- (225.66,244.88) -- (225.66,309.83) -- (160.72,309.83) -- cycle ;
\draw  [draw opacity=0][fill={rgb, 255:red, 137; green, 233; blue, 134 }  ,fill opacity=0.47 ] (236.61,242.11) -- (280.44,242.11) -- (280.44,285.94) -- (236.61,285.94) -- cycle ;
\draw  [fill={rgb, 255:red, 0; green, 0; blue, 0 }  ,fill opacity=1 ] (256.24,264.02) .. controls (256.24,262.76) and (257.26,261.74) .. (258.52,261.74) .. controls (259.79,261.74) and (260.81,262.76) .. (260.81,264.02) .. controls (260.81,265.29) and (259.79,266.31) .. (258.52,266.31) .. controls (257.26,266.31) and (256.24,265.29) .. (256.24,264.02) -- cycle ;
\draw  [color={rgb, 255:red, 250; green, 6; blue, 6 }  ,draw opacity=1 ] (226.05,231.55) -- (291,231.55) -- (291,296.5) -- (226.05,296.5) -- cycle ;
\draw  [draw opacity=0][fill={rgb, 255:red, 137; green, 233; blue, 134 }  ,fill opacity=0.47 ] (301.27,229.44) -- (345.11,229.44) -- (345.11,273.27) -- (301.27,273.27) -- cycle ;
\draw  [fill={rgb, 255:red, 0; green, 0; blue, 0 }  ,fill opacity=1 ] (320.9,251.36) .. controls (320.9,250.09) and (321.93,249.07) .. (323.19,249.07) .. controls (324.45,249.07) and (325.48,250.09) .. (325.48,251.36) .. controls (325.48,252.62) and (324.45,253.64) .. (323.19,253.64) .. controls (321.93,253.64) and (320.9,252.62) .. (320.9,251.36) -- cycle ;
\draw  [color={rgb, 255:red, 250; green, 6; blue, 6 }  ,draw opacity=1 ] (290.72,218.88) -- (355.66,218.88) -- (355.66,283.83) -- (290.72,283.83) -- cycle ;

\draw (175,221) node [scale=0.5]  {$A$};
\draw (242,206) node [scale=0.5]  {$B$};
\draw (304,196) node [scale=0.5]  {$C$};
\draw (188,286) node [scale=0.5]  {$D$};
\draw (251.5,271) node [scale=0.5]  {$E$};
\draw (320,258) node [scale=0.5]  {$F$};

\end{tikzpicture}\]
\caption{Plotting $6$ towers in $K_{\infty}$ for $(t,2)$.}
    \label{fig:(t,2) king lattice}
\end{figure}

Now, consider an arbitrary point $(a,b)$ in the king's lattice. If $(a,b)$ is a point and there exists $v$, where $v$ is a tower, such that $d_K(v,(a,b)) \leq t-2$, then by Remark \ref{kings_broadcast_zone} we know that $(a,b)$ receives signal of at least 2 from the vertex $v$. Consider a point $(a,b)$ such that $d_K(v, (a,b)) = t-1$ then $(a,b)$ receives a signal strength of $1$ from the tower at vertex $v$. This means that in order for the vertex $(a,b)$ to receive enough signal it must lie on a second tower's broadcast zone. We claim such a tower exists, and call it $v'$. We must show that $d_K(v',(a,b)) = t-1$.

By symmetry and to simplify the argument we assume $v_0 = (0,0)$. This implies that $v_1, v_2, v_3,$ and $v_4$ are placed at $(-1, 2t-r), (2t-r, 1), (1, -(2t-r)),$ and $(-(2t-r), 1)$, respectively.


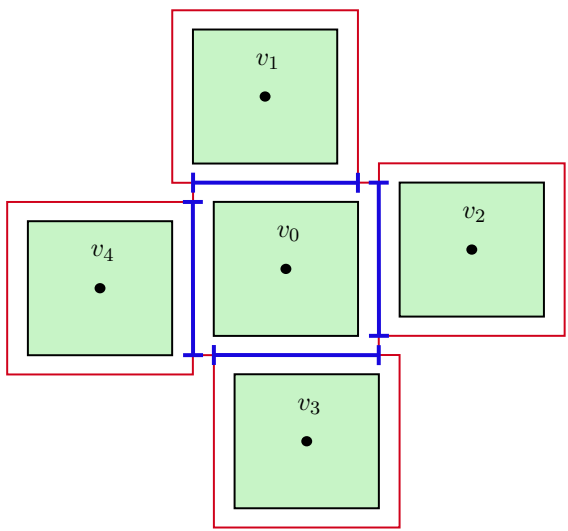
\begin{figure}[htb!]
    \centering
\tikzset{every picture/.style={line width=0.75pt}} 

\begin{tikzpicture}[x=0.75pt,y=0.75pt,yscale=-.75,xscale=.75]

\draw  [color={rgb, 255:red, 0; green, 0; blue, 0 }  ,draw opacity=1 ][fill={rgb, 255:red, 137; green, 233; blue, 134 }  ,fill opacity=0.47 ] (288.55,263.86) -- (385.57,263.86) -- (385.57,354.05) -- (288.55,354.05) -- cycle ;
\draw  [color={rgb, 255:red, 0; green, 0; blue, 0 }  ,draw opacity=1 ][fill={rgb, 255:red, 0; green, 0; blue, 0 }  ,fill opacity=1 ] (334.07,308.96) .. controls (334.07,307.42) and (335.41,306.17) .. (337.06,306.17) .. controls (338.72,306.17) and (340.06,307.42) .. (340.06,308.96) .. controls (340.06,310.49) and (338.72,311.74) .. (337.06,311.74) .. controls (335.41,311.74) and (334.07,310.49) .. (334.07,308.96) -- cycle ;
\draw  [color={rgb, 255:red, 208; green, 2; blue, 27 }  ,draw opacity=1 ] (274.63,250.91) -- (399.5,250.91) -- (399.5,367) -- (274.63,367) -- cycle ;
\draw  [color={rgb, 255:red, 0; green, 0; blue, 0 }  ,draw opacity=1 ][fill={rgb, 255:red, 137; green, 233; blue, 134 }  ,fill opacity=0.47 ] (274.55,147.86) -- (371.57,147.86) -- (371.57,238.05) -- (274.55,238.05) -- cycle ;
\draw  [color={rgb, 255:red, 0; green, 0; blue, 0 }  ,draw opacity=1 ][fill={rgb, 255:red, 0; green, 0; blue, 0 }  ,fill opacity=1 ] (320.07,192.96) .. controls (320.07,191.42) and (321.41,190.17) .. (323.06,190.17) .. controls (324.72,190.17) and (326.06,191.42) .. (326.06,192.96) .. controls (326.06,194.49) and (324.72,195.74) .. (323.06,195.74) .. controls (321.41,195.74) and (320.07,194.49) .. (320.07,192.96) -- cycle ;
\draw  [color={rgb, 255:red, 208; green, 2; blue, 27 }  ,draw opacity=1 ] (260.63,134.91) -- (385.5,134.91) -- (385.5,251) -- (260.63,251) -- cycle ;
\draw  [color={rgb, 255:red, 0; green, 0; blue, 0 }  ,draw opacity=1 ][fill={rgb, 255:red, 137; green, 233; blue, 134 }  ,fill opacity=0.47 ] (260.55,31.86) -- (357.57,31.86) -- (357.57,122.05) -- (260.55,122.05) -- cycle ;
\draw  [color={rgb, 255:red, 0; green, 0; blue, 0 }  ,draw opacity=1 ][fill={rgb, 255:red, 0; green, 0; blue, 0 }  ,fill opacity=1 ] (306.07,76.96) .. controls (306.07,75.42) and (307.41,74.17) .. (309.06,74.17) .. controls (310.72,74.17) and (312.06,75.42) .. (312.06,76.96) .. controls (312.06,78.49) and (310.72,79.74) .. (309.06,79.74) .. controls (307.41,79.74) and (306.07,78.49) .. (306.07,76.96) -- cycle ;
\draw  [color={rgb, 255:red, 208; green, 2; blue, 27 }  ,draw opacity=1 ] (246.63,18.91) -- (371.5,18.91) -- (371.5,135) -- (246.63,135) -- cycle ;
\draw  [color={rgb, 255:red, 0; green, 0; blue, 0 }  ,draw opacity=1 ][fill={rgb, 255:red, 137; green, 233; blue, 134 }  ,fill opacity=0.47 ] (399.55,134.86) -- (496.57,134.86) -- (496.57,225.05) -- (399.55,225.05) -- cycle ;
\draw  [color={rgb, 255:red, 0; green, 0; blue, 0 }  ,draw opacity=1 ][fill={rgb, 255:red, 0; green, 0; blue, 0 }  ,fill opacity=1 ] (445.07,179.96) .. controls (445.07,178.42) and (446.41,177.17) .. (448.06,177.17) .. controls (449.72,177.17) and (451.06,178.42) .. (451.06,179.96) .. controls (451.06,181.49) and (449.72,182.74) .. (448.06,182.74) .. controls (446.41,182.74) and (445.07,181.49) .. (445.07,179.96) -- cycle ;
\draw  [color={rgb, 255:red, 208; green, 2; blue, 27 }  ,draw opacity=1 ] (385.63,121.91) -- (510.5,121.91) -- (510.5,238) -- (385.63,238) -- cycle ;
\draw  [color={rgb, 255:red, 0; green, 0; blue, 0 }  ,draw opacity=1 ][fill={rgb, 255:red, 137; green, 233; blue, 134 }  ,fill opacity=0.47 ] (149.55,160.86) -- (246.57,160.86) -- (246.57,251.05) -- (149.55,251.05) -- cycle ;
\draw  [color={rgb, 255:red, 0; green, 0; blue, 0 }  ,draw opacity=1 ][fill={rgb, 255:red, 0; green, 0; blue, 0 }  ,fill opacity=1 ] (195.07,205.96) .. controls (195.07,204.42) and (196.41,203.17) .. (198.06,203.17) .. controls (199.72,203.17) and (201.06,204.42) .. (201.06,205.96) .. controls (201.06,207.49) and (199.72,208.74) .. (198.06,208.74) .. controls (196.41,208.74) and (195.07,207.49) .. (195.07,205.96) -- cycle ;
\draw  [color={rgb, 255:red, 208; green, 2; blue, 27 }  ,draw opacity=1 ] (135.63,147.91) -- (260.5,147.91) -- (260.5,264) -- (135.63,264) -- cycle ;
\draw [color={rgb, 255:red, 24; green, 8; blue, 221 }  ,draw opacity=1 ][line width=1.5]    (260.5,147.91) -- (260.63,251) ;
\draw [shift={(260.63,251)}, rotate = 269.93] [color={rgb, 255:red, 24; green, 8; blue, 221 }  ,draw opacity=1 ][line width=1.5]    (0,6.71) -- (0,-6.71)   ;
\draw [shift={(260.5,147.91)}, rotate = 269.93] [color={rgb, 255:red, 24; green, 8; blue, 221 }  ,draw opacity=1 ][line width=1.5]    (0,6.71) -- (0,-6.71)   ;
\draw [color={rgb, 255:red, 24; green, 8; blue, 221 }  ,draw opacity=1 ][line width=1.5]    (385.5,251) -- (274.63,250.91) ;
\draw [shift={(274.63,250.91)}, rotate = 360.05] [color={rgb, 255:red, 24; green, 8; blue, 221 }  ,draw opacity=1 ][line width=1.5]    (0,6.71) -- (0,-6.71)   ;
\draw [shift={(385.5,251)}, rotate = 360.05] [color={rgb, 255:red, 24; green, 8; blue, 221 }  ,draw opacity=1 ][line width=1.5]    (0,6.71) -- (0,-6.71)   ;
\draw [color={rgb, 255:red, 24; green, 8; blue, 221 }  ,draw opacity=1 ][line width=1.5]    (371.5,135) -- (260.63,134.91) ;
\draw [shift={(260.63,134.91)}, rotate = 360.05] [color={rgb, 255:red, 24; green, 8; blue, 221 }  ,draw opacity=1 ][line width=1.5]    (0,6.71) -- (0,-6.71)   ;
\draw [shift={(371.5,135)}, rotate = 360.05] [color={rgb, 255:red, 24; green, 8; blue, 221 }  ,draw opacity=1 ][line width=1.5]    (0,6.71) -- (0,-6.71)   ;
\draw [color={rgb, 255:red, 24; green, 8; blue, 221 }  ,draw opacity=1 ][line width=1.5]    (385.5,134.91) -- (385.63,238) ;
\draw [shift={(385.63,238)}, rotate = 269.93] [color={rgb, 255:red, 24; green, 8; blue, 221 }  ,draw opacity=1 ][line width=1.5]    (0,6.71) -- (0,-6.71)   ;
\draw [shift={(385.5,134.91)}, rotate = 269.93] [color={rgb, 255:red, 24; green, 8; blue, 221 }  ,draw opacity=1 ][line width=1.5]    (0,6.71) -- (0,-6.71)   ;

\draw (338.92,285.15) node [scale=0.9]  {$v_{3}$};
\draw (324.92,169.15) node [scale=0.9]  {$v_{0}$};
\draw (310.92,53.15) node [scale=0.9]  {$v_{1}$};
\draw (449.92,156.15) node [scale=0.9]  {$v_{2}$};
\draw (199.92,182.15) node [scale=0.9]  {$v_{4}$};

\end{tikzpicture}
    \caption{Five towers in king's lattice.}
    \label{fig:5towers}
\end{figure}

Consider Figure \ref{fig:5towers}. Let $(a,b)$ be a point where $d_K(v_0,(a,b)) = t-1$. Then, we know that $(a,b)$ lies on the border of the broadcast zone for $v_0$. If this is the case then there exists a unique tower $v'$ where $d_K(v',(a,b)) = t-1$. Assume without loss of generality, that $v_4$ is the tower of distance $t-1$ from $(a,b)$. This means that $(a,b)$ lies on the border of the broadcast zone of $v_4$. Then $(a,b)$ lies on the blue line between $v_0$ and $v_4$. Then it follows that $(a,b)$ will receive 1 signal from tower $v_0$ and tower $v_4$, which will give $(a,b)$ a total reception of 2, since $(a,b)$ is not distance $t-1$ from any other towers broadcast zone. 
\end{proof}

\section{Open Problems}
We now present several open questions on the $(t,r)$ broadcast domination number of other graphs, including cycles and trees, which arose naturally in the course of our work.

We begin with a corollary to Theorem \ref{thm:trpath}. 

\begin{corollary} \label{cor:cycle}
Let $C_n$ be a cycle on $n$ vertices. Then \[\gamma_{t,r}(C_n) \leq\left \lceil \frac{n+r-1}{2t-r} \right \rceil.\]
\end{corollary}

In light of this result it is natural to ask the following.

\begin{question}
Let $C_n$ be a cycle on $n$ vertices. Can the upper bound \[\gamma_{t,r}(C_n) \leq\left \lceil \frac{n+r-1}{2t-r} \right \rceil\] given by Corollary \ref{cor:cycle} be improved?
\end{question}
 
This question was answered by Herrman and van Hintum in \cite{HERRMAN2021270}.
 
Theorem \ref{thm:trpath} allows us to give an upper bound for the $(t,r)$ broadcast domination number of trees using path decompositions. We first give the definition of the path decomposition of a graph and then state our result below.

\begin{definition}
A \textbf{path decomposition} of a graph $G$ is a collection $\Psi$ of edge-disjoint subgraphs $H_1, H_2, ...,H_k$ of $G$ such that every edge of $G$ belongs to exactly one $H_i$, and that every $H_i$ is a path. 
\end{definition}

\begin{corollary}
\label{cor:treedecomp}
Let $\Psi$ be a path decomposition of a tree $T$, with $\Psi = \{H_1, H_2, \ldots, H_n\}$. Then \[\gamma_{t,r}(T) \leq \sum_{i=1}^n \gamma_{t,r} (H_i).\]
\end{corollary}
\begin{proof}

Let $T$ be a tree with path decomposition $\Psi = \{H_1, H_2, H_3, ...,H_n\}$. Note that by Theorem \ref{thm:trpath} each path $H_i \in \Psi$ has a minimum $(t, r)$ broadcast domination set. By placing such a minimum $(t,r)$ broadcast domination set on each path in $\Psi$, we are able to produce a $(t,r)$ broadcast domination set for $T$ (although not necessarily minimum). Thus, we have that 
\[\gamma_{t,r}(T) \leq \gamma_{t,r}(H_1)+\gamma_{t,r}(H_2)+ \cdots + \gamma_{t,r}(H_n) = \sum_{i=1}^n \gamma_{t,r} (H_i). \qedhere\]
\end{proof}

Just as for cycles, it remains open whether this upper bound can be improved.

\begin{question}
Given a tree $T$. Can the upper bound \[\gamma_{t,r}(T) \leq \sum_{i=1}^n \gamma_{t,r} (H_i)\] given by Corollary \ref{cor:treedecomp} be improved?
\end{question}

In this paper, we established an upper bound for the domination number of any $3$D grid graph as stated in Theorem \ref{thm:upperbound}. We now ask the following.

\begin{question}
Given a $3$D grid graph $G_{m,n,k}$. Can the upper bound $\displaystyle \gamma_{t,r} (G_{m,n,k}) \leq 2 \cdot B$ given in Theorem \ref{thm:upperbound} be improved?
\end{question}

\section{Appendix}\label{ap:algorithm}
Here we present the algorithm for generating the $(t,r)$ dominating set for a path on $n$ vertices. The algorithm is written in SageMath and can be found at: \\ \texttt{https://github.com/nzcrepeau/t\_r\_path\_domination.git}.

\bibliographystyle{plain}
\bibliography{bibliography}

\section{Statements and Declarations}
This work was supported in part by the Alfred P. Sloan Foundation, the Mathematical Sciences Research Institute, and the National Science Foundation (grant. No. DMS-1156499). The authors have no relevant financial or non-financial interests to disclose. All authors contributed to the research and the writing of the manuscript.
\end{document}